\documentclass[12pt]{amsart}
\usepackage{amssymb, a4wide, mathdots,url,hyperref, shuffle}
\usepackage{bbm}
\usepackage{tikz}
\usepackage{ytableau}
\usepackage{pb-diagram}

\newcommand{\infl}{\mathrm{infl}}

\newcommand{\Hom}{\operatorname{Hom}}
\renewcommand{\subset}{\subseteq}

\newcommand{\cont}{\operatorname{cont}}

\newcommand{\Vect}{Vec}

\newtheorem{theorem}{Theorem}[section]
\newtheorem{lemma}[theorem]{Lemma}
\newtheorem{proposition}[theorem]{Proposition}
\newtheorem{remark}[theorem]{Remark}

\newtheorem{definition}[theorem]{Definition}
\newtheorem{corollary}[theorem]{Corollary}

\date{\today}

\newcommand{\gotM}{\mathfrak{m}}

\DeclareMathOperator{\irr}{Irr}

\newcommand{\pairs}{\mathfrak{A}}
\newcommand{\sgn}{\operatorname{sgn}}

\newcommand{\smlr}{\ll}

\newcommand{\supp}{\operatorname{supp}}

\newcommand{\Seg}{\operatorname{Seg}}

\newcommand{\m}{\mathfrak{m}}
\newcommand{\n}{\mathfrak{n}}
\newcommand{\la}{\mathfrak{l}}
\newcommand{\Vien}{\mathcal{K}}

\newcommand{\depth}{\mathfrak{d}}
\newcommand{\Mult}{\mathfrak{M}}
\newcommand{\Z}{\mathbb{Z}}

\renewcommand{\subset}{\subseteq}

\newcommand{\RSK}{\mathcal{RSK}}

\newcommand{\id}{\operatorname{Id}}

\newcommand{\Lad}{\operatorname{Lad}}
\newcommand{\rshft}[1]{\overset{\rightarrow}{#1}\vphantom{#1}}              

\newcommand{\wt}{\mathrm{wt}}

\newcommand{\nmod}{\mathrm{mod}}
\newcommand{\fgt}{\mathrm{fgt}}

\newcommand{\gmod}{\mathrm{gmod}}

\newcommand{\KR}{\mathrm{KR}}

\newcommand{\rres}{\mathrm{Res}}
\newcommand{\iind}{\mathrm{Ind}}

\newcommand{\girr}{\mathrm{gIrr}}

\newcommand{\seg}{\mathrm{Seg}}

\begin{document}

\title[Specht as a derivative of RSK]{Graded Specht modules as Bernstein-Zelevinsky derivatives of the RSK model}

\begin{abstract}
We clarify the links between the graded Specht construction of modules over cyclotomic Hecke algebras and the RSK construction for quiver Hecke algebras of type $A$, that was recently imported from the setting of representations of $p$-adic groups.

For that goal we develop a theory of crystal derivative operators on quiver Hecke algebra modules, that categorifies the Berenstein-Zelevinsky strings framework on quantum groups, and generalizes a graded variant of the classical Bernstein-Zelevinsky derivatives from the $p$-adic setting.

Graded cyclotomic decomposition numbers are shown to be a special subfamily of the wider concept of RSK decomposition numbers.

\end{abstract}

\author{Maxim Gurevich}
\address{Department of Mathematics, Technion -- Israel Institute of Technology, Haifa, Israel.}
\email{maxg@technion.ac.il}

\maketitle

\section{Introduction}


The negative part of a quantum group, that is attached to a simple Lie algebra, is categorified by the Grothendieck ring arising from finite-dimensional modules over quiver Hecke algebras (\cite{KLR1,KLR2,rouq,vv}). In Lie type $A$, the resulting graded quiver Hecke algebras are known \cite{brun-kles,rouq} to be nearly isomorphic to affine Hecke algebras (for a generic parameter $q$). The representation theory of the latter family is tightly related to that of $p$-adic general linear groups, which is known for its arithmetic interest.


Our interest lies in the interrelated structures of either of those non-semi-simple abelian categories. In this work we mainly deal with the monoidal category
\[
\mathcal{D}=  \oplus_{\beta\in Q_+} R(\beta)-\gmod\;,
\]
where $R(\beta),\,\beta\in Q_+$, parameterize all quiver Hecke algebras (over a field of characteristic zero) of (finite) type $A$, and $R(\beta)-\gmod$ stands for the category of finite-dimensional graded modules over the algebra.

\subsubsection*{Graded BZ derivatives}
Categorified versions of quantum group crystal operators are a fundamental tool in the study of quiver Hecke algebra modules (\cite{KLR1,MR2822211,kkko-mon}). On the other hand, Bernstein-Zelevinsky derivatives \cite{BZ1,Zel} are among the most classical tools in the study of $p$-adic groups representation theory. One intermediate goal of our current study is to point out an immediate link between those two types of functors.

More precisely, we study categorical versions of some further developed aspects of the crystal theory for quantum groups. Namely, we take the Berenstein-Zelevinsky strings framework \cite{bz-strings,beren-z2} into the setting of modules in $\mathcal{D}$. We are then motivated by the $p$-adic approach of \cite{Zel} to study the monoidal structure on $\mathcal{D}$ through the family of newly defined derivative\footnote{Not to be confused with the familiar notion of `derived functors'. Let us stress that, although much of our discussion deals with abelian categories, no derived categories will appear in this article. Our choice of the words `derived' and `derivative' for terminology stems from the analogy to the classical literature on representations of $p$-adic groups.}
functors. One impetus for this approach was an investigation of the similar notion of extremal words from \cite{klesh-ext,MR1994959}.

We eventually observe (Section \ref{sect:compare}) that the original Bernstein-Zelevinsky derivatives of the $p$-adic setting correspond, by the standard equivalences, to certain useful special cases of the derivative functors on $\mathcal{D}$, that we call \textit{BZ-derivatives}. We show that the Kleshchev-Ram classification \cite{kr2} of simple modules in $\mathcal{D}$, which may be viewed as a graded variant of the Zelevinsky classification for $p$-adic general linear groups, behaves well with respect to the new graded variants of Bernstein-Zelevinsky derivatives.

\subsubsection*{Spearheaded modules and the RSK construction}

One curious notion recently developed in \cite{kk19} is that of a \textit{normal sequence} $L_1,\ldots, L_k\in \mathcal{D}$ of \textit{square-irreducible} (or, real) modules. This is a sequence of simple modules satisfying certain compatibility conditions, which assure that the resulting product (in the sense of the monoidal structure on $\mathcal{D}$) module $L_1\circ \cdots \circ L_k$, even though reducible, enjoys favorable properties.

We call this class of modules \textit{spearheaded}. Those are finite-dimensional (graded) modules $M$, for which
\begin{enumerate}
  \item The head (maximal semi-simple quotient) $L$ of $M$ is simple and self-dual.
  \item The graded multiplicity of $L$ in the Jordan-H\"{o}lder series of $M$ is $1$.
  \item The graded multiplicity of any simple self-dual module $L'\not\cong L$ in the Jordan-H\"{o}lder series of $M$ has positive degrees.
\end{enumerate}

In \cite{gur-klrrsk} it was shown that normal sequences in $\mathcal{D}$ indeed give rise to spearheaded product modules.

We stress that the degree-positivity property, proved in arguments of \cite{kkko-mon}, algebraically encapsulates non-trivial properties attained from a geometric realization of the category in terms of perverse sheaves on quiver varieties. Thus, each normal sequence may be viewed as a potentially insightful model for realizing the simple isomorphism class given by the head of its product module.

One canonical recipe for producing a wide family of normal sequences is the Robinson-Schensted-Knuth construction, designed by the author and Lapid in \cite{gur-lap}.

Briefly, each (up to a shift grading) simple module in $L \in \mathcal{D}$ may be described through the Kleshchev-Ram classification as $L=L_\m$, by combinatorial data known as a multisegment $\m \in \Mult$. This is a categorification of the Lusztig PBW construction for the canonical basis.

A multisegment, viewed as a tuple of pairs of integers, may be transformed under the RSK algorithm into a semistandard bi-tableau $\RSK(\m)$ filled with the same collection of integers. It was shown in \cite{gur-klrrsk} that the resulting bi-tableau gives rise to a normal sequence of \textit{homogeneous} (that is, having a grading concentrated at a single degree) modules, whose resulting spearheaded module $\Gamma(\m)$ realizes $L_\m$ as its head.

In a suitable sense, the RSK standard modules $\{\Gamma(\m)\}_{\m\in \Mult}$ categorify an alternative basis for the negative part of the quantum group, while providing minimal homogeneous models for the corresponding simple modules $\{L_\m\}_{\m\in \Mult}$.

\subsubsection*{Derived RSK modules}

Given an admissible sequence $\mathbf{i}\in \mathbb{Z}^t$ (viewed as a sequence of simple roots of the defining Lie data), the crystal derivative functors define a module $\theta_{\mathbf{i}}(M)\in \mathcal{D}$, for any $M\in \mathcal{D}$, by taking its highest derivative in an appropriate sense.

Theorem \ref{thm:der-nrm} gives a combinatorial condition, in terms of $\mathbf{i}$-strings, for a given normal sequence $L_1,\ldots,L_k$ to produce a new normal sequence $\theta_{\mathbf{i}}(L_1),\ldots, \theta_{\mathbf{i}}(L_k)$. Thus, a zoo of new spearheaded modules may now be produced out of the previously known recipes for normal sequences.

Moreover, the monoidal nature of our derivative operations (Corollary \ref{cor:gr-mul}) makes the Jordan-H\"{o}lder series of the newly constructed modules readable out of the input normal sequence.

In the special above-mentioned case of BZ-derivatives $\theta_{\mathbf{i}} = \theta_{BZ}$ (that is, a canonical choice of $\mathbf{i}$ we call a \textit{BZ-sequence}) the situation is especially appealing. Theorems \ref{thm:bz-final} and \ref{thm:mult-bz} give an explicit combinatorial description, in terms of the Kleshchev-Ram classification, for the spearheaded module that is obtained by deriving any given normal sequence with the $\theta_{BZ}$ operator.

In fact, the motivating case for the development of these tools was the RSK construction. Applying the BZ-derivative operator on a normal sequence defined by the RSK transform $\RSK(\m)$ of any multisegment $\m\in \Mult$, may now produce a \textit{derived RSK module} $\Gamma'(\m)\in \mathcal{D}$. Theorem \ref{thm:head-der} now shows that $\Gamma'(\m)$ is a spearheaded module whose head is given as $L_{\m'}$, where $\m \mapsto \m'$ is a simple truncation operation.

Reasoning in reverse, given a simple module $L= L_\m\in \mathcal{D}$, we may consider the various multisegments $\m_{\partial} \in \Mult$ satisfying $(\m_{\partial})' = \m$. It thus follows that each choice of $\m_{\partial}$ produces its own derived RSK module $\Gamma(\m, \partial):=\Gamma'(\m_{\partial})$ that serves as a separate homogenous model for the same simple $L_{\m}$.

This class was in fact studied (under the name of \textit{enhanced} RSK standard modules) in the ungraded $p$-adic setting in \cite[Section 5]{gur-lap}.

\subsection{Cyclotomic perspective}

Cyclotomic Hecke algebras (of type $A$), or Ariki-Koike algebras, are a family of finite-dimensional algebras, that may be viewed as quotients of the affine Hecke algebra of type $A$. Their representation theory is widely studied, often motivated by analogies with the modular representation theory of permutation groups. Our quiver Hecke algebras of interest, $R(\beta)$, $\beta\in Q_+$, have similarly defined graded finite-dimensional quotient algebras $R(\beta)^\Lambda$ ($\Lambda \in P_+$ belongs to a Lie-theoretic parameter set), which were famously shown by Brundan-Kleshchev \cite{brun-kles} to be isomorphic (as ungraded algebras) to their corresponding cyclotomic Hecke algebras.

\subsubsection*{Specht modules}

The simple $R(\beta)^\Lambda$-modules are known, by Ariki \cite{ariki-aff}, to categorify the images of the dual canonical basis elements under the action of the quantum group on a \textit{Fock space} defined by $\Lambda$.

While the Kleshchev-Ram multisegment description fits the PBW parametrization of quantum group bases, standard bases for Fock spaces have a separate nature. The combinatorial gadget parameterizing those standard bases are $\ell$-\textit{multipartitions}, that is, tuples of $\ell$ integer partitions, where $\ell$ is the height parameter of $\Lambda$.

Brundan-Kleshchev \cite{bk-decomp} showed that standard bases of $\Lambda$-Fock spaces are categorified, in a suitable sense, by the class of \textit{graded Specht modules}, $S_{\Lambda}(\underline{\mu})$, with $\underline{\mu}$ ranging over the collection of $\ell$-multipartitions. Those are modules over graded cyclotomic Hecke algebras, constructed in \cite{bkw,univ-specht} as graded analogues of the modules obtained from a cellular structure in \cite{djm}.

Building upon classifications of Ariki \cite{ariki-class} and Grojnowski \cite{groj}, it was also shown in \cite{bk-decomp} that for multipartitions $\underline{\mu}$ that satisfy a $\Lambda$-\textit{restricted} (or simply, $\Lambda$-\textit{Kleshchev}) combinatorial condition, the resulting graded Specht modules $S_{\Lambda}(\underline{\mu})$ are spearheaded, in our sense. The heads $D_{\Lambda}(\underline{\mu})$ of the $\Lambda$-restricted Specht modules are in bijection with the set of isomorphism classes of simple self-dual modules of cyclotomic Hecke algebras defined by $\Lambda$.

Since all simple modules in $\mathcal{D}$ factor through large enough cyclotomic quotients, the multipartition classification may be pulled onto the quiver Hecke algebra level by inflating Specht modules $\widehat{S}_{\Lambda}(\underline{\mu})\in \mathcal{D}$ and, in the restricted case, their simple heads $\widehat{D}_{\Lambda}(\underline{\mu})$, from the respective quotient algebras.

This point of view was formalized, for example, in \cite{MR2811321} to give an early classification of simple modules in $\mathcal{D}$, while further consistencies of the crystal structures involved were explored in \cite{MR2822211}.
\subsubsection*{Linking the three constructions}

Given an isomorphism class of a simple self-dual module $L\in R(\beta)-\gmod$, we see several canonical methods of explicitly realizing it. In the Kleshchev-Ram method, $L$ becomes a quotient of a \textit{proper standard} module $\KR(\m)$, constructed out of a multisegment $\m$. In the Specht construction, a cyclotomic quotient $R(\beta)^{\Lambda}$ is first chosen, through which $L$ factors, and then $L$ is found as a quotient of a $\Lambda$-restricted Specht module $\widehat{S}_{\Lambda}(\underline{\mu})$.

Lastly, the RSK model finds $L$ as a quotient of another graded module $\Gamma(\m)$ constructed out of the bi-tableau $\RSK(\m)$.

Our main results aim to clarify the relations between those classifications, and specifically, between the three types of modules, $\Gamma(\m)$, $\KR(\m)$ and $\widehat{S}_{\Lambda}(\underline{\mu})$, in $\mathcal{D}$.

Working in the ungraded setting of affine Hecke algebras, Vazirani \cite{MR1923974} made a dictionary that attaches a multisegment $\m= \m(\underline{\mu},\Lambda)\in \Mult$ to a $\Lambda$-restricted multipartition $\underline{\mu}$, so that the (analgous ungraded) isomorphism classes $L_{\m}$ and $\widehat{D}_{\Lambda}(\underline{\mu})$ become equal.

Given such corresponding pair $\m \leftrightarrow \underline{\mu}$, we make the following identification of models.

\begin{theorem}[Theorem \ref{thm:specht-rsk}]
A derived RSK module $\Gamma(\m, \partial)$ exists, which is isomorphic to the inflation of the graded Specht module $\widehat{S}_{\Lambda}(\underline{\mu})$, twisted by a natural sign duality.
\end{theorem}

The decomposition numbers, that is the (graded) multiplicities of simple modules as subquotients of (restricted) Specht modules, are an aspect of high interest in the cyclotomic setting. The categorification work of \cite{bk-decomp} may be seen as proving a variant of the Lascoux--Leclerc--Thibon conjecture. Our Theorems \ref{thm:specht-rsk} and \ref{thm:spear-rsk} now state that the Specht decomposition numbers are a special case of the wider family of \textit{RSK decomposition numbers}, i.e. the graded multiplicities of simple subquotients inside RSK standard modules.

Our result also implies (Corollary \ref{cor:specht-normalcyc}) that, taking the product description of Specht modules as in \cite{univ-specht}, the restricted condition on multipartitions manifests as the normal sequence condition of quiver Hecke algebra modules. 

We note (Corollary \ref{cor:col-rem}) that BZ-derivatives on inflated Specht modules correspond to a combinatorial operation that resembles the column removal that was studied in \cite{MR3533560,MR2006423}.

\subsubsection*{Cyclicity}

In order to incorporate the Kleshchev-Ram model in our comparison, we say that a normal sequence $L_{\m_1},\ldots, L_{\m_k}$ is \textit{cyclic} (inspired by a similar notion from the representation theory of quantum affine algebras \cite{hernan-cyc} and a recent work of the author with M\'{i}nguez \cite{minggur-cyclic} in the $p$-adic setting), when the simple head of its product module is given by $L_{\m_1 + \ldots + \m_k}$.

Corollary \ref{cor:specht-normalcyc} implies that any (sign twist of an) inflated restricted Specht module is a product arising from a cyclic normal sequence.

Consequently, we prove in Proposition \ref{prop:regular-quot} that all (sign twists of) inflated restricted Specht modules may be produced as quotient modules of proper standard modules.

This phenomenon was noted implicitly in \cite{MR1923974} in the ungraded context. The cyclicity property may now give an essential reason for its occurrence, while providing a curious categorical link between the dual PBW basis of a quantum group and the standard bases on its irreducible highest weight modules (viewed as subspaces of the corresponding Fock spaces). 

Through the lens of the higher level Schur-Weyl duality in the sense of \cite{bk-degn}, such link may be put in the context of relations between Verma modules in the classical Bernstein-Gelfand-Gelfand category $\mathcal{O}$ and their parabolic analogues.

\subsection*{Acknowledgements}

The idea of incorporating derivative operations into the RSK setting was first kindly suggested to us by Alberto M\'{i}nguez. I would also like to thank Alexander Kleshchev for sharing his invaluable expertise, Liron Speyer for guidance on the Specht construction and Travis Scrimshaw for interesting discussions on the crystal-theoretic point of view.

Special thanks to Inna Entova-Aizenbud and Ryo Fujita for their crucial input on Lemma \ref{lem:catg}, which is key to Section \ref{sect:padic}. Thanks are also due to David Hernandez and Univerist\'{e} de Paris for their hospitality and support.

This research is supported by the Israel Science Foundation (grant No. 737/20).

\section{Background}

Let us recall the basics of the representation theory of quiver Hecke algebras of type $A_\infty$ (See \cite{gur-klrrsk} and the references within). 


We take the Cartan datum $(\mathcal{I},\cdot)$ as a set labelled by integers $\mathcal{I}=(\alpha_i)_{i\in\mathbb{Z}}$ (simple roots), and an integer valued symmetric bilinear form $\alpha,\beta\mapsto (\alpha,\beta)$ on the free abelian group $Q = \mathbb{Z}[\mathcal{I}]$ (root lattice) given by
\[
(\alpha_i,\alpha_j) = \left\{ \begin{array}{ll} 2 & i=j \\ -1 & |i-j|=1 \\ 0 & |i-j|>1 \end{array}\right.\;, \quad \forall i,j\in \mathbb{Z}\;.
\]
We denote the positive cone $Q_+ = \sum_{\alpha\in \mathcal{I}} \mathbb{Z}_{\geq 0} \alpha\subset Q$. For $\beta_1,\beta_2\in Q_+$, we write $\beta_1\leq \beta_2$ whenever $\beta_2-\beta_1\in Q_+$.

For a positive integer $N$, we denote the subcone $Q_+^{(N)} = \sum_{-N\leq i\leq N} \mathbb{Z}_{\geq 0} \alpha_i\subset Q_+$.

Let $\beta=\sum_{\alpha\in \mathcal{I}} c_\alpha \alpha\in Q_+$ be fixed. Its height is defined as $|\beta|=\sum_{\alpha\in \mathcal{I}}c_\alpha\in \mathbb{Z}$.

We denote the finite set of tuples
\[
\mathcal{I}^\beta = \{ \nu =(\nu_1,\ldots, \nu_{|\beta|})\in \mathcal{I}^{|\beta|}\;:\; \nu_1 + \ldots + \nu_{|\beta|} = \beta\}\;.
\]

The \textit{quiver Hecke algebra} (or Khovanov-Lauda-Rouquier algebra) related to $\beta$ is defined to be the unital associative complex algebra $R(\beta)$, which is generated
by $\{\mathfrak{e}(\nu)\}_{\nu \in \mathcal{I}^{ \beta  }}$, $\{y_1,\ldots, y_{|\beta|}\}$, $\{\psi_1, \ldots, \psi_{|\beta|-1}\}$, subject to the relations
\[
\mathfrak{e}(\nu)\mathfrak{e}(\nu') = \left\{\begin{array}{ll} \mathfrak{e}(\nu) & \nu = \nu' \\ 0 & \nu \neq \nu'\end{array}\right.\;,\quad \sum_{\nu\in  \mathcal{I}^{ \beta  }} \mathfrak{e}(\nu) = 1\;,
\]
\[
y_i \mathfrak{e}(\nu) = \mathfrak{e}(\nu) y_i\;,\quad \psi_i \mathfrak{e}(\nu) = \mathfrak{e}(s_i\cdot\nu)\psi_i\;,\;\forall i\;,
\]
\[
y_iy_j = y_iy_j\;,\;\forall i,j\;,\quad \psi_i\psi_j = \psi_j \psi_i\;,\mbox{ for }|i-j|>1\;,
\]
\[
y_j\psi_i = \psi_i y_j\;, \mbox{ for }j\not\in \{i,i+1\}\;,
\]
\[
(y_{i+1}\psi_i - \psi_i y_{i})\mathfrak{e}(\nu) =  \left\{\begin{array}{ll} \mathfrak{e}(\nu) & \nu_i = \nu_{i+1} \\ 0 & \nu_i \neq \nu_{i+1} \end{array}\right.\;,\quad (y_{i}\psi_{i} - \psi_{i} y_{i+1})\mathfrak{e}(\nu) =  \left\{\begin{array}{ll} -\mathfrak{e}(\nu) & \nu_i = \nu_{i+1} \\ 0 & \nu_i \neq \nu_{i+1} \end{array}\right.\;,
\]
\[
(\psi_{i+1}\psi_i\psi_{i+1} - \psi_i\psi_{i+1}\psi_i)\mathfrak{e}(\nu) = \left\{\begin{array}{ll} \mathfrak{e}(\nu) & (\nu_i, \nu_{i+1}, \nu_{i+2}) = (\alpha_t, \alpha_{t+1}, \alpha_t),\mbox{ for }t\in\mathbb{Z} \\ -\mathfrak{e}(\nu) & (\nu_i, \nu_{i+1}, \nu_{i+2}) = (\alpha_t, \alpha_{t-1}, \alpha_t),\mbox{ for }t\in\mathbb{Z} \\ 0 & \mbox{otherwise} \end{array}\right.,
\]
\[
\psi_i^2\mathfrak{e}(\nu) = \left\{\begin{array}{ll} (y_i-y_{i+1})\mathfrak{e}(\nu) & (\nu_i, \nu_{i+1}) = (\alpha_t, \alpha_{t+1}),\mbox{ for }t\in\mathbb{Z} \\ -(y_i-y_{i+1})\mathfrak{e}(\nu) & (\nu_i, \nu_{i+1}) = (\alpha_t, \alpha_{t-1}),\mbox{ for }t\in\mathbb{Z} \\  0 & \nu_i = \nu_{i+1} \\ \mathfrak{e}(\nu) & \mbox{otherwise} \end{array}\right.\;.
\]

Here, $s_i\cdot \nu\in \mathcal{I}^\beta$ denotes an action of a simple transposition, i.e. a switch of $\nu_i$ with $\nu_{i+1}$.

The algebra $R(\beta)$ becomes ($\mathbb{Z}$-)graded, when setting the degrees
\[
\deg(\mathfrak{e}(\nu)) = 0,\quad \deg(y_i)  = 2,\quad \deg(\psi_i\mathfrak{e}(\nu)) = -(\nu_i, \nu_{i+1})
\]
on the generators.

We write $R(\beta)-\nmod$ ($R(\beta)-\gmod$) for the abelian category of (graded) finite dimensional left modules over $R(\beta)$.

For a graded module $M = (M_i)_{i\in \mathbb{Z}}\in R(\beta)-\gmod$ and an integer $k$, we write
\[
M\langle k \rangle  = (M_{i-k})_{i\in \mathbb{Z}}\in R(\beta)-\gmod
\]
to be the shifted module.



\subsection{Restriction and induction}\label{sec:resind}

Given $\underline{\beta} = (\beta_1,\ldots, \beta_k)\in (Q_+)^k$, we set the graded algebra
\[
R(\underline{\beta}) = R(\beta_1)\otimes \cdots \otimes R(\beta_k)\;.
\]

Setting $i(\underline{\beta}) = \beta_1 + \ldots +\beta_k$, we have a natural embedding of algebras $\iota_{\underline{\beta}}: R(\underline{\beta}) \to R(\iota(\underline{\beta}))$, as in \cite[Section 2.2]{MR2822211}.

For $\nu_i\in \mathcal{I}^{\beta_i}$, $i=1,\ldots,k$, we have the natural concatenation operation $\nu_1\ast\cdots\ast \nu_k\in \mathcal{I}^{i(\underline{\beta})}$. We then obtain an idempotent element
\[
\mathfrak{e}(\underline{\beta}):= \sum_{\nu_i\in \mathcal{I}^{\beta_i}\, i=1,\ldots,k} \mathfrak{e}(\nu_1\ast\cdots\ast \nu_k) \in R(   i(\underline{\beta}))\;.
\]
Evidently, $\iota_{\underline{\beta}}(1) = \mathfrak{e}(\underline{\beta})$ holds, for the identity element $1\in R(\underline{\beta})$.

Thus, the embedding of algebras gives rise to an exact restriction functor
\[
\rres_{\underline{\beta}}: R(i(\underline{\beta}))-\gmod \;\to\; R(\underline{\beta})-\gmod\;,\quad \rres_{\underline{\beta}}(M) = \mathfrak{e}(\underline{\beta})M\;.
\]


The restriction functor commits a left-adjoint induction functor
\[
\iind_{\underline{\beta}}: R(\underline{\beta})-\gmod \;\to\; R(i(\underline{\beta}))-\gmod\;,\quad
\iind_{\underline{\beta}}(M) = R(i(\underline{\beta}))\otimes_{\iota_{\underline{\beta}}(R(\underline{\beta}))}M\;.
\]


Given $M_i \in R(\beta_i)-\gmod$, for $i=1,\ldots k$, we write the induction operation as a convolution product
\[
M_1\circ \cdots \circ M_k = \iind_{\underline{\beta}}(M_1 \otimes \cdots \otimes M_k)\;.
\]
This product equips the larger abelian category
\[
\mathcal{D} = \oplus_{\beta\in Q_+} R(\beta)-\gmod \,
\]
with a monoidal structure.



\subsection{The $(\mathbb{Z}/2\mathbb{Z})^3$ categorical symmetry}

Certain natural symmetries of the category $\mathcal{D}$ follow naturally from definitions of quiver Hecke algebras. Let us sum up these various involutive constructions that are scattered across the relevant literature (e.g. \cite{MR2822211,univ-specht,kkko-mon}).

Let us write $\beta \mapsto \beta^\dagger$ for the additive involution on $Q$ defined on $\mathcal{I}$ by $\alpha_i \mapsto \alpha_{-i}$, for all $i\in \mathbb{Z}$.

For all $\beta\in Q_+$, there is a unique graded algebra isomorphism $\sgn = \sgn_{\beta}: R(\beta) \to R(\beta^\dagger)$ given by
\[
\sgn(\mathfrak{e}(\nu_1,\ldots,\nu_{|\beta|})) =  \mathfrak{e}(\nu_1^\dagger,\ldots,\nu^\dagger_{|\beta|}),\quad \sgn(y_i) = -y_i,\quad \sgn(\psi_i) = -\psi_i
\]
on generators.
Clearly, $\sgn$ is an involution, in the sense that $\sgn_{\beta}^{-1} = \sgn_{\beta^\dagger}$.

We write
\[
M\mapsto M^{\sgn}
\]
for the corresponding functor $R(\beta)-\gmod \to R(\beta^\dagger)-\gmod$.

There is also a unique graded algebra involution $\sigma= \sigma_\beta: R(\beta) \to R(\beta)$ given by
\[
\sigma(\mathfrak{e}(\nu_1,\ldots,\nu_{|\beta|})) =  \mathfrak{e}(\nu_{|\beta|},\ldots,\nu_{1}),\quad \sigma(y_i) = y_{|\beta|+1-i},
\]
\[
\sigma(\psi_i \mathfrak{e}(\nu_1,\ldots,\nu_{|\beta|}))  = \left\{\begin{array}{rl} -\psi_{|\beta|-i} \mathfrak{e}(\nu_{|\beta|},\ldots,\nu_{1}) & \nu_i = \nu_{i+1} \\  \psi_{|\beta|-i} \mathfrak{e}(\nu_{|\beta|},\ldots,\nu_{1}) & \nu_i \neq \nu_{i+1}\end{array}\right.
\]
on generators.

We write
\[
M\mapsto M^{\sigma}
\]
for the corresponding functor $R(\beta)-\gmod \to R(\beta)-\gmod$.

Finally, the algebra $R(\beta)$ also possesses an anti-involution $\tau$ given as an identity on all generators in the definition of the algebra. For $M\in  R(\beta)-\gmod$, the complex dual space $M^\ast$ becomes a graded left $R(\beta)$-module through $\tau$.

Thus, we think of $M\mapsto M^\ast$ as a contravariant involution of the category $R(\beta)-\gmod$.

The three functors, $\sgn, \sigma$ and $\ast$, clearly extend to commuting functorial involutions of the category $\mathcal{D}$.

\begin{proposition}\label{prop:monoid-inv}
For modules $M_1\in R(\beta_1)-\gmod$ and $M_2 \in R(\beta_2)-\gmod$, we have the natural isomorphisms
\[
(M_1\circ M_2)^{\sgn} \cong M_1^{\sgn}\circ M_2^{\sgn},
\]
in $R(\beta_1^\dagger + \beta_2^\dagger)-\gmod$, and
\[
(M_1\circ M_2)^\sigma\cong M_2^\sigma\circ M_1^\sigma, \quad
(M_1\circ M_2)^\ast\langle - (\wt(M_1),\wt(M_2))\rangle \cong M_2^\ast\circ M_1^\ast
\]
in $R(\beta_1+\beta_2)-\gmod$.
\end{proposition}
\begin{proof}
The first two isomorphisms follow from the identities
\[
\sgn_{\beta_1+\beta_2}\iota_{(\beta_1,\beta_2)} = \iota_{(\beta_1,\beta_2)} (\sgn_{\beta_1}\otimes \sgn_{\beta_2}),\;\sigma_{\beta_1+\beta_2}\iota_{(\beta_1,\beta_2)} = \iota_{(\beta_2,\beta_1)} (\sigma_{\beta_2}\otimes \sigma_{\beta_1})\;,
\]
which are a result of their validity on generators.

The last isomorphism is a consequence of \cite[Theorem 2.2]{MR2822211} and is stated, for example, in \cite[Section 2.1.]{kkko-mon}.

\end{proof}

\subsection{Simple modules and Kleshchev-Ram classification}\label{sec:kr-class}

Let $\irr(\beta)$ ($\girr(\beta)$) be the set of isomorphism classes of simple modules in $R(\beta)-\nmod$ ($R(\beta)-\gmod$).

Each $M\in \girr(\beta)$ remains simple as an ungraded module, while each $L\in \irr(\beta)$ has a unique, up to shift, graded structure (e.g. \cite[Section 2.4]{bk-decomp}).

Moreover, for each $M\in \girr(\beta)$, it is known that there is a (unique) integer $k$, such that $(M\langle k \rangle)^\ast \cong  M\langle k\rangle$.

Thus, we will often treat $\irr(\beta)$ as a subset of $\girr(\beta)$, that is, the isomorphism classes of self-dual simple modules in $R(\beta)-\gmod$.

Given $M\in R(\beta)-\gmod$, we write $[M]\in \mathbb{Z}_{\geq 0} [\girr(\beta)]$ as a formal sum of the Jordan-H\"{o}lder series of $M$. Taking shifts into account, we may write it as a sum
\[
[M]= \sum_{L\in \irr(\beta)} \sum_{i\in \mathbb{Z}} m_{L,i} [L\langle i \rangle]\;.
\]
Given any $L\in \irr(\beta)$ and $M\in R(\beta)-\gmod$, we define the \textit{graded multiplicity} of $L$ in $M$ as the Laurent polynomial
\[
m(M,L)(q) = \sum_{i\in \mathbb{Z}} m_{L,i} q^i \in \mathbb{Z}_{\geq 0 }[q,q^{-1}]\;.
\]

Let us recall the classification of
\[
\irr_{\mathcal{D}} := \sqcup_{\beta\in Q_+} \irr(\beta)
\]
as obtained in \cite{kr2}. 

For each pair of integers $i\leq j$, we denote the element $\Delta(i,j) = \alpha_i + \alpha_{i+1} + \ldots +\alpha_j\in Q_+$. We refer to $\seg = \{\Delta(i,j)\}_{i\leq j} \subset Q_+$ as the set of \textit{segments} (positive roots).

For a segment $\Delta = \Delta(i,j)\in \seg$, we write $i = b(\Delta)$ and $j= e(\Delta)$ for its begin and end points.

Let $\leq$ denote the total lexicographical order on $\seg$, so that $\Delta_1\leq \Delta_2$, if $b(\Delta_1)< b(\Delta_2)$ holds, or that both $b(\Delta_1)= b(\Delta_2)$ and $e(\Delta_1)\leq e(\Delta_2)$ hold.

Similarly, let $\leq_r$ denote the right lexicographical order on $\seg$, defined as $\leq$, but with the roles of $b(\Delta)$ and $e(\Delta)$ reversed.

We refer to elements of the free abelian monoid
\[
\Mult = \mathbb{Z}_{\geq0}[\seg]
\]
as \textit{multisegments}.


There is a natural additive map $\wt: \Mult \to Q_+$, defined by $\wt(\Delta)= \Delta$, for each $\Delta\in \Seg$\footnote{Through a slight abuse of notation, we put $\Seg$ as a subset of both $Q_+$ and $\Mult$.}.

Each $\Delta\in \seg$ is attached with a \textit{segment module} 
$L_\Delta \in \irr(\Delta)$. It may be characterized as the unique self-dual $1$-dimensional $R(\Delta)$-module, for which there exists $\nu=(\nu_1, \ldots, \nu_{|\Delta|})\in \mathcal{I}^{\Delta}$ with $\nu_1 = \alpha_{e(\Delta)}$, so that $\mathfrak{e}(\nu)L_{\Delta}\neq0$.

Each $\m \in\Mult$ can be uniquely written as $\m = \sum_{i=1}^k p_i \Delta_i$, for segments $\Delta_1 <_r \ldots <_r \Delta_k$ in $\seg$.

In these terms, the Kleshchev-Ram classification attaches to $\m\in \Mult$ the \textit{proper standard module}
\[
\KR(\gotM) := L_{\Delta_1}^{\circ p_1}\circ \cdots \circ L_{\Delta_k}^{\circ p_k}\;\left\langle {p_1 \choose 2} + \ldots + {p_k \choose 2}\right\rangle\in R(\wt(\m))-\gmod \;.
\]
Here $L^{\circ p} = L\circ \cdots \circ L$ denotes the $p$-th induction product of a module with itself.

\begin{theorem}\label{thm:kr}\cite[Theorem 7.2]{kr2}
The head (or co-socle, i.e. maximal semisimple quotient) of $\KR(\m)$, denoted as $L_{\gotM}$, is simple and self-dual. The resulting map
\[
\Mult \to \irr_{\mathcal{D}}\;,\quad \m\mapsto L_{\m}
\]
is a bijection.
\end{theorem}

For $\m\in \Mult$, we set
\[
\mathfrak{b}(\m) = n_1\alpha_{b(\Delta_1)} + \ldots +  n_k\alpha_{b(\Delta_k)}\;,
\]
and write $|\m| = |\mathfrak{b}(\m)|$, that is, the number of segments used to define $\m$.

For $L= L_\m\in \irr_{\mathcal{D}}$, we also write $\mathfrak{b}(L)=\mathfrak{b}(\m)$.

It will also be convenient to write $\wt(M) = \beta$, for any module $M\in R(\beta)-\gmod$. In that sense, we obtain $\wt(L_{\m}) = \wt(\m)$.

\begin{remark}\label{rem:kl-inv}
The Kleshchev-Ram construction of a proper standard module out of a given multisegment $\m \in \Mult$ is less canonical than how it may appear in our current presentation. More precisely, the construction depends on a choice of a total order on the set $\mathcal{I}$ (i.e. on $\mathbb{Z}$).\footnote{The more general procedure as extended in \cite{MR3205728} depends on a choice of a (convex) order on the set of positive roots, that is, on $\Seg$.} Our definition of $\KR(\m)$ takes the order given by $\alpha_i > \alpha_j$, for $i<j$.

In fact, when applying the algebra involution $\sigma$, the resulting module $\KR(\m)^\sigma$ would produce the proper standard module given by the reverse order on $\mathcal{I}$, that is, when $\alpha_i < \alpha_j$, for $i<j$. Hence, observing that $L_{\m}^\sigma$ is the self-dual simple head of $\KR(\m)^\sigma$, we see that the map $\m \mapsto L^\sigma_{\m}$ gives another bijection $\Mult \to \irr_{\mathcal{D}}$, still within the Kleshchev-Ram framework.

When comparing these categories with representations of $p$-adic groups (see Section \ref{sect:padic}), the involution $\sigma$ becomes a realization of the Zelevinsky involution on irreducible representations.

\end{remark}

Note now, that $\Seg$ as a subset of $Q$ is invariant under the $\dagger$ involution. More precisely, $\Delta(i,j)^\dagger = \Delta(-j,-i)\in \Seg$, for any $i\leq j$. Thus, we may extend it to an additive involution $\dagger : \Mult \to \Mult$.

\begin{proposition}\label{prop:twist-irr}
For all $\m\in \Mult$, we have isomorphisms
\[
\KR(\m)^{\sgn} \cong \KR(\m^\dagger)^\sigma ,\quad L^{\sgn}_{\m} \cong  L^\sigma_{\m^\dagger}
\]
 of modules in $R(\wt(\m)^\dagger)-\gmod$.
\end{proposition}
\begin{proof}
For $\Delta\in \Seg$ and $\nu\in \mathcal{I}^{\Delta}$ with $\mathfrak{e}(\nu)L_{\Delta}\neq0$, we have a $1$-dimensional module $L^{\sgn \circ \sigma}_{\Delta}$ in $R(\Delta^\dagger)-\gmod$, with $(\mathfrak{e}(\nu))^{\sgn \circ\sigma}L^{\sgn \circ\sigma}_{\Delta}\neq\{0\}$. By definitions of $\sigma$ and $\sgn$, $(\mathfrak{e}(\nu))^{\sgn \circ\sigma} = \mathfrak{e}(\mu)$, for $\mu = (\mu_1,\ldots,\mu_{|\Delta|})\in \mathcal{I}^{\Delta^\dagger}$ with $\mu_1 = \alpha_{b(\Delta)^\dagger} = \alpha_{e(\Delta^\dagger)}$. Hence, $L^{\sgn \circ\sigma}_{\Delta}\cong L_{\Delta^\dagger}$.

Now, since $\Delta_1 <_r \Delta_2$ holds for segments $\Delta_1,\Delta_2\in \Seg$, if and only if, $\Delta_2^\dagger <_r \Delta_1^\dagger$ holds, we see from Proposition \ref{prop:monoid-inv} that $\KR(\m)^{\sgn\circ \sigma} \cong \KR(\m^\dagger)$ and the other identities evidently follow.

\end{proof}

\subsection{Cyclotomic quotients}

Consider the monoid $P = \mathbb{Z}[\mathcal{I}^\vee]$, with $\mathcal{I}^\vee = (\Lambda_i)_{i\in\mathbb{Z}}$, as a lattice dual to $Q$. In other words, a bi-additive pairing $P\times Q \to \mathbb{Z}$, $\Lambda, \beta\mapsto (\Lambda,\beta)$, is fixed by setting $(\Lambda_i, \alpha_j) = \delta_{i,j}$, for all $i,j\in \mathbb{Z}$.

We denote the positive cone $P_+ = \mathbb{Z}_{\geq0}[\mathcal{I}^\vee]\subset P$. For an element $\Lambda = \sum_{i\in \mathbb{Z}} c_i \Lambda_i\in P_+$, we denote its \textit{level} $|\Lambda| = \sum_{i\in \mathbb{Z}} c_i\in \mathbb{Z}_{\geq0}$.

A tuple $\kappa = (k_1,\ldots,k_l)\in \mathbb{Z}^l$, with $k_1\geq \ldots \geq k_l$, is called a \textit{multicharge}. Each such multicharge $\kappa$ describes a weight $\Lambda(\kappa) = \Lambda_{k_1} + \ldots + \Lambda_{k_l} \in P_+$ with $|\Lambda(\kappa)| = l$.

An involution $\Lambda\mapsto \Lambda^\dagger$ may be defined on $P$, similarly to $Q$. For a multicharge $\kappa = (k_1,\ldots,k_l)$, we write $\kappa^\dagger = (-k_l,\ldots, -k_1)$ for the multicharge satisfying $\Lambda(\kappa)^\dagger = \Lambda(\kappa^\dagger)$.

Given $\beta\in Q_+$ and $\Lambda\in P_+$, the \textit{cyclotomic quiver Hecke algbera} $R(\beta)^{\Lambda}$ is defined to be the quotient of the algebra $R(\beta)$ by the ideal generated by homogeneous elements of the form
\[
y_1^{(\Lambda, \nu_1)} \mathfrak{e}(\nu),\quad\mbox{for all }\nu = (\nu_1,\ldots,\nu_{|\beta|})\in\mathcal{I}^\beta\;.
\]

The resulting algebras $R(\beta)^{\Lambda}$ are finite-dimensional and graded.

For any $d\in \mathbb{Z}_{\geq1}$ and $\Lambda\in P_+$, it was shown \cite{brun-kles} that the algebra
\[
R^{\Lambda}_d= \oplus_{\beta\in Q_+,\; |\beta|=d} R(\beta)^{\Lambda}
\]
is isomorphic to the cyclotomic Hecke algebra (of type $A$) defined by $\Lambda$.

For each $\beta\in Q_+$ and $\Lambda\in P_+$, the isomorphism $\sgn_\beta$ factors through the cyclotomic quotients to produce an involutive isomorphism
\[
\sgn = \sgn_\beta^\Lambda : R(\beta)^{\Lambda} \to R(\beta^\dagger)^{\Lambda^\dagger}
\]
of graded algebras.

We write $M\mapsto M^{\sgn}$ for the corresponding functor $R(\beta)^\Lambda-\gmod \to R(\beta^\dagger)^{\Lambda^\dagger}-\gmod$.

\begin{remark}\label{rem:cycl}
The involution $\sigma_\beta$ on $R(\beta)$ does not factor through the cyclotomic quotients. Yet, some literature on quiver Hecke algebras (see e.g. \cite{MR3771147}) gives an alternative definition for those algebras as quotients by ideals of elements of the form $y_{|\beta|}^{(\Lambda, \nu_{|\beta|})} \mathfrak{e}(\nu)$. Evidently, the algebra involution $\sigma_\beta$ then induces a canonical isomorphism between the quotients produced by the two different definitions.

\end{remark}

Let us denote the natural inflation functor
\[
\infl : R(\beta)^{\Lambda} - \gmod\,\to\, R(\beta)-\gmod\;,
\]
defined by pulling module structure through the graded quotient homomorphism $R(\beta)\to R(\beta)^{\Lambda}$.

The anti-involution $\tau$ on $R(\beta)$ factors through the quotient to an anti-involution on $R(\beta)^{\Lambda}$. Hence, a duality functor $M \mapsto M^\ast$, that commutes with inflation, may be defined similarly on $R(\beta)^{\Lambda} - \gmod$.

In particular, same as for $\irr(\beta)$, we let $\irr(\beta)^{\Lambda}$ simultaneously denote the set of isomorphism classes of ungraded simple modules of $R(\beta)^{\Lambda}$, or of self-dual graded simple modules.

The inflation functor gives a natural embedding
\[
\infl : \irr(\beta)^{\Lambda}\to \irr(\beta)\;.
\]

\section{Derivatives of graded modules}

Let us first recall the notion crystal operators (or Kashiwara operators) for quiver Hecke algebras.

Given a simple root $\alpha = \alpha_i\in \mathbb{Z}$ and an integer $n\in \mathbb{Z}_{>0}$, there is a unique simple self-dual class $L(i,n)\in \irr(n \alpha)$. As a graded $R(n\cdot \alpha)$-module, $L(i,n)$ has a projective cover $P(i,n)$.

For $\beta\in Q_+$ with $n \alpha\leq \beta$, as for example in \cite[Section 2.6]{bk-decomp}, we define the \textit{divided power functor}
\[
\theta_i^{(n)}: R(\beta)-\gmod \to R(\beta- n \alpha)-\gmod
\]
by
\[
\theta_i^{(n)}(M) = \Hom_{R(n \alpha)-\nmod)}(P(i,n), \rres_{(\beta- n \alpha, n \alpha)}(M))\;,
\]
where $R(n \alpha)$ is taken as a subalgebra of $R(\beta- n \alpha,\beta)\cong R(\beta- n \alpha) \otimes R(n \alpha)$.

In other words, when non-zero, we have
\begin{equation}\label{eq:expl}
\rres_{(\beta- n \alpha_i, n \alpha_i)}(M) =  \theta_i^{(n)}(M)\boxtimes L(i,n)\;.
\end{equation}

We naturally extend these into exact functors $\theta_i^{(n)}: \mathcal{D}\to \mathcal{D}$ by setting $\theta_i^{(n)}|_{R(\beta)-\gmod}$ as the zero functor, whenever $n\alpha\not\leq \beta$.

Following standard notation, for any $M\in \irr_{\mathcal{D}}$, we let $\epsilon_i(M)$ be the maximal $k\in \mathbb{Z}_{\geq0}$, for which $\theta_i^{(k)}(M)\neq\{0\}$ holds. It is known (\cite[Section 10.1]{kkko-mon}), that $\theta_i^{(\epsilon_i(M))}(M)$ is always a self-dual simple module.

We call a sequence $\mathbf{i}= (i_1,\ldots,i_t)\in \mathbb{Z}^t$ of integers \textit{admissible}, when $i_r\neq i_{r+1}$ holds, for all $1\leq r<t$.

Let us fix one admissible sequence $\mathbf{i}= (i_1,\ldots,i_t)$ for the time being.

Let us consider the additive monoid $A_t = {\mathbb{Z}_{\geq0}}^t$. We have a monoid homomorphism
\[
A_t \to Q_+,\quad \underline{a} = (a_1,\ldots, a_t)\,\mapsto\, \beta(\mathbf{i},\underline{a})\;,
\]
given by $\beta(\mathbf{i},\underline{a}) = a_1\alpha_{i_1} + \ldots + a_t\alpha_{i_t}\in Q_+$.

For any $\underline{a} = (a_1,\ldots, a_t)\in A_t$, we denote the composition of functors
\[
\theta^{\underline{a}}_{\mathbf{i}} = \theta_{i_t}^{(a_t)}\circ \cdots \circ \theta_{i_1}^{(a_1)}: \mathcal{D}\to \mathcal{D}\;.
\]

We see that $\theta^{\underline{a}}_{\mathbf{i}}$ restricts to a functor from $R(\beta)-\gmod$ to $R(\beta- \beta(\mathbf{i},\underline{a}))-\gmod$, for all $\beta(\mathbf{i},\underline{a})\leq \beta$ in $Q_+$.

Next, we equip the monoid $A_t$ with a bi-additive non-symmetric form $(\,,\,)_{\mathbf{i}}: A_t\times A_t \to \mathbb{Z}$ given by

\[
(e_r,e_u)_{\mathbf{i}} = \left\{\begin{array}{lll}  1 & r = u \\ (\alpha_{i_r}, \alpha_{i_u}) & r > u \\ 0 & u < r \end{array} \right.\;,
\]
on the standard monoid generators $e_1 = (1,0,\ldots,0), e_2 = (0,1,0,\ldots, 0), \ldots, e_t = (0,\ldots ,0,1)$.

Note, that
\[
(\underline{a}_1, \underline{a}_2)_{\mathbf{i}} + (\underline{a}_2, \underline{a}_1)_{\mathbf{i}} = ( \beta(\mathbf{i},\underline{a}_1),\beta(\mathbf{i},\underline{a}_2))
\]
holds, for all $\underline{a}_1,\underline{a}_2\in A_t$.

We now turn to an investigation of the monoidal properties of the derivative functors $\theta^{\underline{a}}_{\mathbf{i}}$.

\begin{proposition}\label{prop:filt-der}

Given modules $M_j\in R(\beta_j)-\gmod$, $j=1,\ldots, s$, an admissible sequence $\mathbf{i}\in \mathbb{Z}^t$ and $\underline{a}\in A_t$, the graded module
\[
\theta^{\underline{a}}_{\mathbf{i}}(M_1\circ\cdots \circ M_s)
\]
has a filtration of submodules, whose composition factors are indexed by decompositions $\underline{a}= \underline{a}_1 + \ldots+ \underline{a}_s$, with $\underline{a}_1,\ldots, \underline{a}_s\in A_t$, and are given as
\[
\theta^{\underline{a_1}}_{\mathbf{i}}(M_1)\circ \cdots \circ \theta^{\underline{a_s}}_{\mathbf{i}}(M_s) \langle \Phi_{\beta_1,\ldots,\beta_s}(\mathbf{i}, \underline{a}_1, \ldots, \underline{a}_s )\rangle\;,
\]
where
\[
\Phi_{\beta_1,\ldots,\beta_s}(\mathbf{i}, \underline{a}_1, \ldots, \underline{a}_s ) = \sum_{1\leq j < k \leq s} \left( (\underline{a}_j, \underline{a}_k)_{\mathbf{i}} - (\beta_k, \beta(\mathbf{i},\underline{a}_j))\right)\in \mathbb{Z}\;.
\]

\end{proposition}
\begin{proof}
Here we use the Mackey theory as formulated in \cite[Proposition 2.1]{gur-klrrsk} (see also \cite[Proposition 2.18]{KLR1}, \cite[Theorem 2.23]{klesh-ext}), for decomposing
\[
N = \rres_{(\beta_1+ \ldots+\beta_s - \beta(\mathbf{i},\underline{a}),a_t\alpha_{i_t}, \ldots, a_1\alpha_{i_1}) }(M_1\circ\cdots\circ M_s)\;.
\]
Using the identity \eqref{eq:expl}, we see that a filtration for $N$ exists whose composition factors are given as
\[
N_{\underline{a}_1,\ldots, \underline{a}_s} = (\theta^{\underline{a_1}}_{\mathbf{i}}(M_1)\circ \cdots \circ \theta^{\underline{a_s}}_{\mathbf{i}}(M_s)) \boxtimes (\circ_{j=1}^s L(i_t, a^j_t) ) \boxtimes \cdots \boxtimes (\circ_{j=1}^s L(i_1, a^j_1) ) \langle -d(\underline{a}_1,\ldots, \underline{a}_s)\rangle\;,
\]
where $\underline{a}_j = (a^j_1,\ldots,a^j_t)\in A_t$, $j=1,\ldots,s$, are such that $\underline{a}_1+\ldots + \underline{a}_s = \underline{a}$, and $d(\underline{a}_1,\ldots, \underline{a}_s)$ is the integer explicitly given in \cite[Proposition 2.1]{gur-klrrsk}.

More specifically,
\[
d(\underline{a}_1,\ldots, \underline{a}_s) = \sum_{1\leq j < k \leq s} \left( \sum_{r=1}^t  ( \beta_k - \beta(\mathbf{i}, \underline{a}_k), a^j_{r}\alpha_{i_r} ) + \sum_{1\leq r < u\leq t}  (a^k_{u}\alpha_{i_u},a^j_{r}\alpha_{i_r} )\right) =
\]
\[
= \sum_{1\leq j < k \leq s} \left( (\beta_k, \beta(\mathbf{i},\underline{a}_j)) - (\beta(\mathbf{i},\underline{a}_k),\beta(\mathbf{i},\underline{a}_j)) +  (\underline{a}_k, \underline{a}_j)_{\mathbf{i}} - \sum_{r=1}^t a^j_r a^k_r\right) =
\]
\[
= \sum_{1\leq j < k \leq s} \left( (\beta_k, \beta(\mathbf{i},\underline{a}_j)) -  (\underline{a}_j, \underline{a}_k)_{\mathbf{i}} - \sum_{r=1}^t a^j_r a^k_r\right)\;.
\]

Finally, recall that
\[
L(i,n) \cong \KR(n\cdot \Delta(i,i))=  (L_{\Delta(i,i)})^{\circ n} \langle {n \choose 2}\rangle\;,
\]
 for all $i\in \mathbb{Z}$ and $n\geq1$. It follows that $L(i,n_1)\circ L(i,n_2) \cong  L(i,n_1+n_2)\left\langle -n_1n_2\right\rangle$.

Recursively, we obtain that
\[
N_{\underline{a}_1,\ldots, \underline{a}_s} \cong
\]
\[
\cong (\theta^{\underline{a_1}}_{\mathbf{i}}(M_1)\circ \cdots \circ \theta^{\underline{a_s}}_{\mathbf{i}}(M_s)) \boxtimes L(i_t, a_t) \boxtimes \cdots \boxtimes L(i_1, a_1) \left\langle -d(\underline{a}_1,\ldots, \underline{a}_s) - \sum_{r=1}^t \sum_{1\leq j< k\leq s} a^j_r a^k_r \right\rangle\;,
\]
finishing our claim.

\end{proof}

Let us take the lexicographic order on $A_t$, that is, $(a_1,\ldots,a_t) < (b_1,\ldots b_t)$, when $a_1 = b_1, \ldots, a_{r-1} = b_{r-1}$ and $a_r< b_r$, for an index $r$.

For $M\in \mathcal{D}$, let $\underline{a} = a(\mathbf{i},M) \in A_t$ be the maximal index relative to this order, for which $\theta_{\mathbf{i}}^{\underline{a}}(M)\neq \{0\}$.

Following \cite{bz-strings}, we call $a(\mathbf{i},M)$ the \textit{$\mathbf{i}$-string} of $M$.

For $M \in R(\beta)-\gmod$, we define $\theta_{\mathbf{i}}(M) = \theta^{a(\mathbf{i},M)}_{\mathbf{i}}(M)\in R(\beta-\beta(\mathbf{i},\underline{a}))-\gmod$ to be the \textit{$\mathbf{i}$-derivative} of $M$.

\begin{lemma}\label{lem:simple-sd}
For any $M\in \irr_{\mathcal{D}}$ and an admissible sequence $\mathbf{i}\in \mathbb{Z}^t$, the module $\theta_{\mathbf{i}}(M)$ is simple and self-dual.

\end{lemma}
\begin{proof}

Writing $\mathbf{i}= (i_1,\ldots, i_t)$ with $\mathbf{i'} = (i_1,\ldots, i_{t-1})$, we see that $\theta_{\mathbf{i}}(M) = \theta_{i_t}^{(\epsilon_{i_t}(N))}(N)$, where $N = \theta_{\mathbf{i'}}(M)$, by definition of the $\mathbf{i}$-derivative.

By induction on $t$, we may assume that $N$ is simple and self-dual. Then, $\theta_{i_t}^{(\epsilon_{i_t}(N))}(N)$ is simple and self-dual by our earlier remark.

\end{proof}

\begin{lemma}\label{lem:add-string}
For given modules $M_1,\ldots,M_s\in \mathcal{D}$, we have the identity
\[
a(\mathbf{i},M_1\circ\cdots\circ M_s) = a(\mathbf{i},M_1) + \ldots + a(\mathbf{i},M_s)\;.
\]
in $A_t$. For a subquotient module $N$ of $M_1\circ\cdots\circ M_s\in \mathcal{D}$, we have $a(\mathbf{i},N)\leq a(\mathbf{i},M_1\circ\cdots\circ M_s)$.
\end{lemma}
\begin{proof}
Suppose that $\underline{a}\in A_t$ is such that $\theta^{\underline{a}}_{\mathbf{i}}(N)\neq \{0\}$, for a subquotient $N$ of $M_1\circ\cdots\circ M_s$.

By exactness, $\theta^{\underline{a}}_{\mathbf{i}}(M_1\circ\cdots\circ M_s)\neq \{0\}$. According to Proposition \ref{prop:filt-der}, there must exist a decomposition $\underline{a} = \underline{a}_1 + \ldots+ \underline{a}_s$, for which the product $\theta^{\underline{a_1}}_{\mathbf{i}}(M_1)\circ \cdots \circ \theta^{\underline{a_s}}_{\mathbf{i}}(M_s)$ is non-zero.

Hence, $\underline{a}_j\leq  a(\mathbf{i},M_j)$, for all $1\leq j\leq s$. It implies that $\underline{a} \leq a(\mathbf{i},M_1) + \ldots + a(\mathbf{i},M_s)$ holds in $A_t$.

Finally, $\theta^{a(\mathbf{i},M_1) + \ldots + a(\mathbf{i},M_s)}_{\mathbf{i}}(M_1\circ\cdots\circ M_s)$ is non-zero again by Proposition \ref{prop:filt-der}.

\end{proof}

\begin{proposition}\label{prop:hst-der}
For given modules $M_1,\ldots, M_s\in \mathcal{D}$ and an admissible sequence $\mathbf{i}\in \mathbb{Z}^t$, we have an isomorphism of graded modules

\[
\theta_{\mathbf{i}}(M_1\circ \ldots \circ M_s) \cong \theta_{\mathbf{i}}(M_1)\circ \cdots \circ \theta_{\mathbf{i}}(M_s)\langle \Phi(\mathbf{i},M_1,\ldots, M_s)\rangle\;,
\]
where
\[
\Phi(\mathbf{i},M_1,\ldots, M_s)= \Phi_{\wt(M_1),\ldots,\wt(M_s)}(\mathbf{i}, a(\mathbf{i},M_1), \ldots, a(\mathbf{i},M_s))\in \mathbb{Z}
\]
is defined using the number from Proposition \ref{prop:filt-der}.

\end{proposition}
\begin{proof}
The statement follows from Proposition \ref{prop:filt-der} and Lemma \ref{lem:add-string} after noting that whenever a decomposition
\[
\underline{a}_1 + \ldots+ \underline{a}_s = a(\mathbf{i},M_1) + \ldots + a(\mathbf{i},M_s)
\]
in the monoid $A_t$ exists with an index $j$, such that $\underline{a}_j \neq a(\mathbf{i},M_j)$ holds, there must be an index $k$, for which an inequality $\underline{a}_k > a(\mathbf{i},M_k)$ holds.

\end{proof}

\begin{corollary}\label{cor:gr-mul}
  For given modules $M_1,\ldots, M_s\in \mathcal{D}$, a simple module $L\in \irr_{\mathcal{D}}$ and an admissible sequence $\mathbf{i}\in \mathbb{Z}^t$, the graded multiplicities satisfy the relation
  \[
  m(\theta_{\mathbf{i}}(M_1)\circ\cdots \circ \theta_{\mathbf{i}}(M_s), \theta_{\mathbf{i}}(L))(q) =
  \]
  \[
  =\left\{\begin{array}{ll} q^{- \Phi(\mathbf{i},M_1,\ldots, M_s)} m(M_1\circ\cdots \circ M_s, L)(q) &
 a(\mathbf{i},L) = a(\mathbf{i},M_1) + \ldots + a(\mathbf{i},M_s) \\ 0 &  a(\mathbf{i},L) \neq  a(\mathbf{i},M_1) + \ldots + a(\mathbf{i},M_s) \end{array} \right.\;.
  \]
\end{corollary}

We note that a property of extremal words in module characters \cite[Proposition 2.31]{klesh-ext} may be viewed as a variant of the last corollary.

\subsection{Bernstein-Zelevinsky derivatives}
Let us write $\mathbf{i}_0 = (T, T-1, \ldots, -T+1, -T)\in \mathbb{Z}^{2T+1}$, for a choice of $T\in \mathbb{Z}_{>0}$. We will say that $\mathbf{i}_0$ is a \textit{BZ-sequence} for $\beta\in Q_+$, if $\beta\in Q_+^{(T)}$.

For a module $M\in R(\beta)-\gmod$ and a BZ-sequence $\mathbf{i}_0$ for $\beta$, we define $\theta_{BZ}(M) = \theta_{\mathbf{i}_0}(M)$ to be the \textit{BZ-derivative} of $M$. It is easily seen to be independent of the particular choice of $\mathbf{i}_0$ (that is, of $T$).

By Lemma \ref{lem:simple-sd}, for $L\in \irr_{\mathcal{D}}$, we have $ \theta_{BZ}(L)\in \irr_{\mathcal{D}}$. The resulting map on classes of simple modules can be explicated as follows.

For a segment $\Delta = \Delta(i,j)\in \Seg$ with $i<j$, we set $\Delta' = \Delta(i+1, j)\in \Seg$. For any $i\in \mathbb{Z}$, we set $\Delta(i,i)' = 0$.

For a segment $\Delta = \Delta(i,j)\in \Seg$, we set $\Delta^+ = \Delta(i-1,j)$.

Given a multisegment $\m = \sum_i \Delta_i \in \Mult$, we write $\m' = \sum_i \Delta'_i\in\Mult$ and $\m^+ = \sum_i \Delta^+_i\in\Mult$.

Clearly, $(\m^+)' = \m$ holds, for all $\m \in \Mult$.

The following lemma is a property discovered through various settings affiliated with our discussion. For representations of $p$-adic linear groups it was explored in \cite{ming-induit} and \cite{jantz-jac} (see \cite[Theorem 5.11]{LM2}), for affine Hecke algebras modules in \cite{MR1923974}, while its decategorified versions in terms of crystal structure on quantum groups were noted in \cite{bz-strings} and \cite{reine}.

We give our own standalone proof rather than involving various equivalences and dualities.

\begin{lemma}\label{lem:cryst-grph}
For a multisegment $\m \in \Mult$ and an integer $j\in \mathbb{Z}$, let us write $\m_j\leq \m$ for the multisegment consisting of all segments $\Delta$ in $\m$ with $b(\Delta)=j$.

Suppose that $\m_{j+1}=0$, for a given $j\in \mathbb{Z}$. Then,
\[
\theta_{(j)}(L_{\m}) \cong L_{\m - \m_j + \m'_j}\;.
\]

\end{lemma}

\begin{proof}
For any segment $\Delta\in \Seg$, it follows from the defining properties of segment modules that
\[
\epsilon_j(L_{\Delta}) = \left\{ \begin{array}{ll} 1 & b(\Delta) = j \\ 0 & b(\Delta)\neq j\end{array}\right.\;.
\]
Using Lemma \ref{lem:add-string}, it then follows that $\epsilon_j(\KR(\m)) = |\m_j|$, and consequently, $\epsilon_j(L_{\m}) \leq |\m_j|$.

Knowing that $\theta_{(j)}(L_{\m})$ must lie in $\irr_{\mathcal{D}}$, it suffices to show that $\rres_{(\wt(\m) - |\m_j| \alpha_j, |\m_j| \alpha_j)}(L_{\m})$ contains the subquotient $L_{\m - \m_j + \m'_j}\boxtimes L(j,|\m_j|)$, up to a shift of grading.

We denote $\n = \m - \m_j + \m'_j$. Due to head simplicity of proper standard modules, the desired fact would follow, once we obtain that the space
\begin{equation}\label{eq:homsp}
\Hom_{R(\wt(\m))-\nmod} (\KR(\n)\circ (L_{\Delta(j,j)})^{\circ |\m_j|}, L_{\m}) \cong
\end{equation}
\[
\cong
\Hom ( \KR(\n)\boxtimes L(j,|\m_j|), \rres_{(\wt(\m) - |\m_j| \alpha_j, |\m_j| \alpha_j)}(L_{\m}))
\]
does not vanish.

By \cite[Lemma 4.2]{gur-klrrsk}, up to a shift of grading, we may identify the module $\KR(\m)$ with
\begin{equation}\label{eq:ungr-1}
\KR(\m_{-T}) \circ \KR(\m_{-T+1}) \circ \cdots \circ \KR(\m_T),
\end{equation}
assuming $\wt(\m) \in Q_+^{(T)}$. Similarly, using the assumption $\m_{j+1}=0$, we may write $\KR(\n)$, up to a shift of grading, as
\[
\KR(\m_{-T}) \circ \KR(\m_{-T+1}) \circ \cdots \circ \KR(\m_{j-1}) \circ \KR(\m'_j) \circ \KR(\m_{j+2})\circ \cdots  \circ \KR(\m_T).
\]
Recall the evident property (\cite[Lemma 4.2]{gur-klrrsk}) that $L_{\Delta(j,j)}$ commutes, relative to the convolution product, with $\KR(\m_r)$, for all $j+1<r$.

Hence, in $R(\wt(\m))-\nmod$, we have
\[
\KR(\n)\circ (L_{\Delta(j,j)})^{\circ |\m_j|} \cong
\]
\begin{equation}\label{eq:ungr-2}
\cong \KR(\m_{-T}) \circ \KR(\m_{-T+1}) \circ \cdots \circ \KR(\m_{j-1}) \circ \KR(\m'_j) \circ (L_{\Delta(j,j)})^{\circ |\m_j|}\circ \KR(\m_{j+2})\circ \cdots  \circ \KR(\m_T).
\end{equation}
Applying Lemma \ref{lem:add-string} in same manner as above, we see that $\KR(\m'_j) \boxtimes (L_{\Delta(j,j)})^{\circ |\m_j|}$ is isomorphic, up to a shift, to $\rres_{(\wt(\m_j) - |\m_j| \alpha_j, |\m_j| \alpha_j)}(\KR(\m_j))$. In particular, we see that the module in \eqref{eq:ungr-1} is a quotient of the one in \eqref{eq:ungr-2}. Thus, we obtain a surjective $R(\wt(\m))-\nmod$ morphism from $\KR(\n)\circ (L_{\Delta(j,j)})^{\circ |\m_j|}$ to $\KR(\m)$, which implies the non-vanishing of \eqref{eq:homsp}.

\end{proof}

\begin{proposition}\label{prop:BZ}
For any multisegment $\m\in \Mult$, we have
\[
\theta_{BZ}(L_{\m}) \cong L_{\m'}\;.
\]

\end{proposition}
\begin{proof}
Suppose that $\wt(\m)\in Q_+^{(T)}$. We can write $\m = \sum_{-T\leq r\leq T} \m_r$, with $\m_r$ as in the statement of Lemma \ref{lem:cryst-grph}.

Recalling that
\[
\theta_{BZ}(L_{\m}) = (\theta_{(-T)}\circ \theta_{(-T+1)}\cdots \circ \theta_{(T)})(L_{\m})\;,
\]
a consecutive application of Lemma \ref{lem:cryst-grph} then gives the identity.
\end{proof}

\begin{proposition}\label{prop:b-homom}
For a multisegment $\m\in\Mult$ a BZ-sequence $\mathbf{i}_0$ for $\wt(\m)$, we have
\[
\beta(\mathbf{i}_0, a(\mathbf{i}_0, L_{\m}))= \mathfrak{b}(\m)\;.
\]
In particular, for $\mathbf{i}_0 = (T, T-1, \ldots, -T+1, -T)$, the map $\m \mapsto a(\mathbf{i}_0, L_{\m})$ is an additive function defined on the monoid
\[
\{\m\in \Mult\;:\;\wt(\m)\in Q_+^{(T)} \}\;,
\]
and $\mathfrak{b}(\m) =\mathfrak{b}(\n)$ implies $a(\mathbf{i}_0, L_{\m}) = a(\mathbf{i}_0, L_{\n})$ for any $\m,\n$ in that monoid.
\end{proposition}
\begin{proof}
The first statement follow from Proposition \ref{prop:BZ} after noting that $\wt(\m)-\wt(\m') = \mathfrak{b}(\m)$.

Other parts of the proposition are a consequence of the evident injectivity of the monoid homomorphism $\underline{a} \mapsto \beta(\mathbf{i}_0, \underline{a})$ and additivity of $\m \mapsto \mathfrak{b}(\m)$.

\end{proof}

Let us define another bi-additive (non-symmetric) integer form on $Q$ by setting
\[
(\alpha_i,\alpha_j)_\ell = \left\{ \begin{array}{ll} 1 & i=j \\ -1 & i=j-1 \\ 0 & j-i\not\in\{0,1\} \end{array}\right.\;, \quad \forall i,j\in \mathbb{Z}\;.
\]

For multisegments $\m_1,\ldots,\m_s\in \Mult$, we define the integer number
\[
\Phi(\m_1,\ldots,\m_s) =\sum_{1\leq j < k\leq s} \left(( \mathfrak{b}(\m_j),\mathfrak{b}(\m_k))_\ell - ( \mathfrak{b}(\m_j),\wt(\m_k))\right)\;.
\]
It will also be useful to take note of the following combinatorial invariants. As in \cite[Section 6.2]{gur-klrrsk}, for any pair of multisegments
\[
\n_1= \sum_{i=1}^{k_1} \Delta^1_i,\quad \n_2= \sum_{i=1}^{k_2} \Delta^2_i\in \Mult\;,
\]
we define the number
\[
C(\n_1,\n_2) = \# \{(i,j)\;:\; b(\Delta^1_i) = e(\Delta^2_j)+1 \}\;.
\]
More generally, for a tuple  $\m_1,\ldots,\m_s\in \Mult$, we write
\begin{equation}\label{eq:c-mults}
C(\m_1,\ldots,\m_s) =\sum_{1\leq j < k\leq s}  C(\m_j,\m_k)\;,
\end{equation}
and
\begin{equation}\label{eq:cc-mults}
C'(\m_1,\ldots,\m_s) =\sum_{1\leq j < k\leq s}  C(\rshft\m_j,\m_k)\;.
\end{equation}

Here, we write $\rshft\Delta = \Delta(i+1,j+1)\in \Seg$, for any $\Delta= \Delta(i,j)\in \Seg$, and extend additively to $\Mult$.

\begin{proposition}\label{prop:combi}
The identities
\[
C(\m_1,\ldots,\m_s)-C'(\m_1,\ldots,\m_s)  = \Phi(\m_1,\ldots,\m_s)= \Phi(\mathbf{i}_0, L_{\m_1},\ldots,L_{\m_s})
\]
hold (with the right-hand side defined in Proposition \ref{prop:hst-der}), for all  $\m_1,\ldots,\m_s\in \Mult$.
\end{proposition}
\begin{proof}
By additivity, it is enough to check that
\[
C(\Delta_1,\Delta_2) - C(\rshft\Delta_1,\Delta_2) = \Phi(\Delta_1,\Delta_2)
\]
holds, for any pair segments $\Delta_1,\Delta_2\in \Seg$ to show the first equality. Indeed, the equality for segments is easily verified.

The second equality follows from Proposition \ref{prop:b-homom} and the monotone nature of BZ-sequences.
\end{proof}

\begin{theorem}\label{thm:mult-bz}
  For any given modules $L_{\m_1},\ldots, L_{\m_s}, L_{\n}\in \irr_{\mathcal{D}}$, the graded multiplicities satisfy the relation
  \[
  m(L_{\m'_1}\circ\cdots \circ L_{\m'_s}, L_{\n'})(q) =
  \]
  \[
  =\left\{\begin{array}{ll} q^{- \Phi(\m_1,\ldots, \m_s)} m(L_{\m_1}\circ\cdots \circ L_{\m_s}, L_{\n})(q) &
 \mathfrak{b}(\n) = \mathfrak{b}(\m_1)+ \ldots + \mathfrak{b}(\m_s) \\ 0 &  \mathfrak{b}(\n) \neq \mathfrak{b}(\m_1)+ \ldots + \mathfrak{b}(\m_s) \end{array} \right.\;.
  \]
\end{theorem}
\begin{proof}
This a special case of Corollary \ref{cor:gr-mul}, when applied on $\theta_{\mathbf{i}_0}$, for a choice of BZ-sequence $\mathbf{i}_0$ for $\wt(\m_1 +\ldots + \m_s)$. The statements are equivalent by Propositions \ref{prop:BZ} and \ref{prop:combi}, while the assumed conditions are equivalent by Proposition \ref{prop:b-homom}.

\end{proof}

\section{Relation with representations of $p$-adic general linear groups}\label{sect:padic}
The aim of this section is to point out that the operators $\theta_{\mathbf{i}}$, for a BZ-sequence $\mathbf{i}$, on the category $\mathcal{D}$, are the direct analogue of the Bernstein-Zelevinsky derivatives. The latter were a central tool in the classical treatment \cite{BZ1,Zel} of the basics of the smooth representation theory of $p$-adic $GL_n$ groups.

This section has a standalone nature and does directly not pertain our results on quiver Hecke algebra modules. Hence, it may be skipped without damaging the main reading narrative.

\subsection{Background}
Let us recall the essentials of the relevant $p$-adic setting and its categorical links with representations of quiver Hecke algebras. We make a concise summary of the more detailed discussion in \cite[Section 3]{gur-klrrsk}.

Let $F$ be a fixed $p$-adic field, and $G_n = GL_n(F)$, for $n\geq1$, be the locally compact (reductive) group of $n\times n$ invertible matrices over it.

We write $G_n-\nmod$ for the abelian category of finite-length (typically infinite-dimensional) smooth representations of $G_n$ over the complex field, and $\irr(G_n)$ for the set of isomorphism classes of irreducible representations in $G_n-\nmod$.

Given a decomposition $n = n_1 + n_2$, we denote the Levi subgroup $M_{n_1,n_2} = G_{n_1}\times G_{n_2} < G_{n_1+n_2}$, consisting of matrices that are block-diagonal with blocks of sizes $(n_1,n_2)$. There is an exact functor, the \textit{Jacquet functor} (parabolic restriction),
\[
\mathbf{r}_{n_1,n_2}: G_{n_1+n_2}-\nmod \to M_{n_1,n_2}-\nmod
\]
serving as the appropriate analogue of a restriction operation.

Its right adjoint functor is the (standard, normalized) \textit{parabolic induction} functor, which we denote as a product
\[
(\pi_1,\pi_2)\mapsto \pi_1\times \pi_2\quad M_{n_1,n_2}-\nmod \to  G_{n_1+n_2}-\nmod\;.
\]

\subsubsection{Derivatives}

Let $U_n< G_n$ be the subgroup of upper unitriangular matrices, and let $\psi: U_n \to \mathbb{C}^\times$ be a fixed non-degenerate character of the group.

The exact Whittaker functor $\mathbf{W} = \mathbf{W}_n: G_n - \nmod \to \Vect$, whose role is omnipresent throughout the theory of representations of $p$-adic groups, is defined by taking the space of $(U_n,\psi)$-co-invariants of a representation.

A classical result \cite{rodier} states the for each $\pi\in \irr(G_n)$, $\mathbf{W}(\pi)$ is either a $1$-dimensional space ($\pi$ being \textit{generic}), or the zero space.

For each $0\leq k \leq n$, The \textit{Bernstein-Zelevinsky derivative} may now be defined as the exact composed functor
\[
\mathfrak{D}_{n,k} = ( \mathbf{W}\otimes \id)\circ \mathbf{r}_{k, n-k} : G_n-\nmod \to G_{n-k}-\nmod\;.
\]
Here, the product of abelian categories $\Vect \times (G_n-\nmod)$ is naturally identified with $G_n- \nmod$.

In particular, the Whittaker functor $\mathbf{W}_n = \mathfrak{D}_{n,n}$ becomes a special case of a Bernstein-Zelevinsky derivative.

\begin{remark}
Our definition is slightly unconventional, when compared to some literature. First, Bernstein-Zelevinsky derivatives are often defined using restriction functors that involve the mirabolic subgroup of $G_n$. An equivalence to our sort of definition is quite straightforward by standard tools.

A more essential point to be made is that a functor of the transposed form $( \id\otimes \mathbf{W})\circ \mathbf{r}_{n-k, k}$ is more commonly encountered. Those two variants are often referred to as left and right derivatives, while their mutual interplay often becomes useful in the $p$-adic literature (e.g. \cite{cs-branch}).

Our current treatment will not require both notions. Instead, we choose to take the variant of definition which would be better consistent with our conventions on derivative operations on the quiver Hecke algebra side.

\end{remark}

For a representation $\pi\in G_n-\nmod$, the \textit{highest derivative} of $\pi$ is set to be the representation $\mathfrak{D}_{BZ}(\pi)= \mathfrak{D}_{n,k_{\max}}(\pi)$, where $k_{\max}$ is the largest integer for which the representation is non-zero.

Finally, we may define the \textit{full derivative} as the exact functor
\[
\mathfrak{D} = \mathfrak{D}_n = \oplus_{k=0}^n \mathfrak{D}_{n,k}
\]
going from $G_n-\nmod$ into $\bigoplus_{k=0}^n G_{n-k} - \nmod$, viewed as a sum of abelian categories. This point of view was considered already in \cite{Zel} on a semisimplified level.

\subsubsection{Decompositions according to supercuspidal support}

A representation $\rho\in \irr(G_d)$ is called \textit{supercuspidal}, when $\mathbf{r}_{d_1,d_2}(\rho)=\{0\}$, for all decompositions $d= d_1+d_2$.

Each such supercuspidal representation $\rho\in \irr(G_d)$ defines a Serre subcategory $\mathcal{C}(\rho,n)$ of $G_{dn}-\nmod$, for each $n\geq1$, by taking the full subcategory of representations whose (isomorphism classes of) irreducible subquotients can be produced as subquotients of representations of the form
\[
\rho\chi^{k_1}\times \cdots \times \rho \chi^{k_n}\;,
\]
where $k_1,\ldots, k_n$ are integers and $\chi^k: G_n \to \mathbb{C}^\times$ denotes the character given by the $p$-adic norm $g \mapsto |\det(g)|_F^k$.

Furthermore, for an irreducible $\pi \in \mathcal{C}(\rho,n)$, the sequence $k_1,\ldots, k_n$ is in fact defined uniquely, up to ordering. In a convenient choice of notation, we may therefore write $\supp_\rho(\pi) = \alpha_{k_1} + \ldots + \alpha_{k_n} \in Q_+$, for the \textit{supercuspidal support} of $\pi$.

Refining our decomposition, we may write $\mathcal{C}_\rho^\beta$, for each $\beta\in Q_+$, for the full subcategory of $\mathcal{C}(\rho, |\beta|)$, consisting of representations, all of whose irreducible subquotients $\pi$ satisfy $\supp_\rho(\pi) = \beta$.

A decomposition of abelian categories is thus attained
\[
\mathcal{C}(\rho,n) =\bigoplus_{\beta\in Q_+, |\beta|=n} \mathcal{C}_\rho^\beta\;,\quad  \sigma = \oplus \sigma_\beta \;.
\]

Equipped with the parabolic induction functor, the category
\[
\mathcal{C}_\rho^{\mathbb{Z}} =  \bigoplus_{n=0}^\infty \mathcal{C}(\rho,n) = \bigoplus_{\beta\in Q_+} \mathcal{C}_\rho^\beta\;,
\]
(ranging over representation categories of groups of various ranks) becomes a monoidal category.

\subsubsection{Bernstein-Rouquier equivalences}

Passing through module categories over affine Hecke algebras, the following equivalence holds.

\begin{theorem}\label{thm:bern-rou}

For each supercuspidal representation $\rho\in \irr(G_d)$ and each weight $\beta\in Q_+$, there is an exact functor
\[
\mathcal{T}_{\rho,\beta}:     R(\beta)-\nmod \to \mathcal{C}^\beta_\rho\;,
\]
which gives an equivalence of abelian categories.

Summing $\mathcal{T}_\rho := \oplus_{\beta\in Q_+} \mathcal{T}_{\rho,\beta}$ gives an equivalence between the (ungraded) sum of categories $\bigoplus_{\beta\in Q_+} R(\beta) - \nmod$ and $\mathcal{C}^{\mathbb{Z}}_\rho$.

The equivalence is monoidal in the sense that
\[
\mathcal{T}_{\rho,\beta_1+\beta_2}(N_1\circ N_2)\cong \mathcal{T}_{\rho,\beta_1}(N_1)\times \mathcal{T}_{\rho,\beta_2}(N_2)
\]
holds, for any modules $N_1\in R(\beta_1)-\nmod$ and $N_2 \in R(\beta_2)-\nmod$.

We also have a functorial isomorphism of the form
\[
\mathbf{r}_{\beta_1,\beta_2}(\mathcal{T}_\rho(N))\cong (\mathcal{T}_{\rho,\beta_2} \otimes \mathcal{T}_{\rho,\beta_1})(\rres_{\beta_2,\beta_1}(N))
\]
(when the Levi subgroups $M_{|\beta_1|,|\beta_2|}$ and $M_{|\beta_2|,|\beta_1|}$ are naturally identified), for any $N \in R(\beta)-\nmod $ and a decomposition $\beta = \beta_1 + \beta_2$ in $Q_+$.

Here, $\mathbf{r}_{\beta_1,\beta_2}$ stands for the composition of $\mathbf{r}_{|\beta_1|,|\beta_2|}$ with the projection $\mathcal{C}(\rho, |\beta_1|) \times\mathcal{C}(\rho, |\beta_2|)  \to \mathcal{C}^{\beta_1}_\rho \times \mathcal{C}^{\beta_2}_\rho$.

\end{theorem}

\begin{proof}
This is a composition of the so-called Bernstein equivalence between $p$-adic groups and affine Hecke algebras, together with the Rouquier equivalence between affine Hecke algebras and quiver Hecke algebras. It was explicated in \cite[Theorem 3.10, Proposition 3.11]{gur-klrrsk}\footnote{The source in fact deals with the inverse equivalence.}.

The last statement is deduced from the preceding monoidality statement by taking adjoint functors. The second adjointness principle for $p$-adic groups implies that $\mathbf{r}_{n_2,n_1}$ is right-adjoint to the transposed functor $(\pi_2,\pi_1) \mapsto \pi_1\times \pi_2$ taking $M_{n_2,n_1}$-representations to $G_{n_1+n_2}$-representations.

\end{proof}

We further recall that the Zelevinsky classification gives a bijection between multisegments in $\Mult$ and the irreducible representations in $\mathcal{C}^{\mathbb{Z}}_\rho$, for any supercuspidal $\rho\in \irr(G_d)$. Namely, for each $\m \in \Mult$ an irreducible representation
\[
Z(\m) = Z_\rho(\m) \in \mathcal{C}_\rho^{\wt(\m)}
\]
is constructed.

Through the equivalences of Theorem \ref{thm:bern-rou}, the Kleshchev-Ram classification serves as a graded generalization of the Zelevinsky classsification in the sense that an isomorphism of irreducible representations
\begin{equation}\label{eq:zel-kr}
Z_\rho(\m)\cong \mathcal{T}_\rho(L_{\m})
\end{equation}
holds.

\subsection{Comparing derivatives}\label{sect:compare}

Let $\partial: Q_+\to \Mult$ be the additive monoid homomorphism sending each $\alpha_i\in \mathcal{I}$ to $\alpha_i = \Delta(i,i)\in \Mult$. So, $\wt(\partial(\gamma))=\gamma$, for all $\gamma\in Q_+$.

For any $\gamma\in Q_+$, we say that $L_{\partial(\gamma)}$ is a (or, the) \textit{generic} representation in $\irr(\gamma)$.

From the point of the view of representations of $p$-adic groups, this terminology makes sense because of the following well-known classification of irreducible representations with a Whittaker model.

\begin{proposition}\label{prop:rodier}
For a multisegment $\m\in \Mult$, the irreducible representation $Z_\rho(\m)\in \mathcal{C}_\rho^{\mathbb{Z}}$ has a non-zero (i.e. $1$-dimensional) Whittaker space $\mathbf{W}(Z_{\rho}(\m))$, if and only if, $\m = \partial(\wt(\m))$.
\end{proposition}

We would first like to compare the Whittaker functor with our construction of derivatives in $\mathcal{D}$. To facilitate such comparison let us consider the following general property of abelian categories of the kind appearing in our setting.

\begin{lemma}\label{lem:catg}
Let $L_0\in \irr(\beta)$ be a given simple graded module, for $\beta\in Q_+$.

Suppose that two exact functors 
\[
\mathcal{F}, \mathcal{G}: R(\beta)-\gmod \to Vec
\]
are given, for which 
\[
\dim \mathcal{F}(L) = \dim \mathcal{G}(L) = \left\{\begin{array}{ll} 1 & L\cong L_0 \\ 0 & L\not\cong L_0 \end{array}\right.
\]
holds, for any $L\in \irr(\beta)$.

Then, $\mathcal{F}$ and $\mathcal{G}$ are isomorphic functors.

\end{lemma} 

\begin{proof}

We make use of the cyclotomic quotients structure of $R(\beta)$. Let $\{\Lambda^n\}_{n=1}^\infty$ be a sequence in $P_+$, so that the action of $R(\beta)$ on $L_0$ factors through $R(\beta)^{\Lambda^1}$, and that for each $\alpha\in \mathcal{I}$, $(\Lambda^n , \alpha)$ is an increasing sequence (going to infinity).

It follows that each morphism space in $R(\beta)-\gmod$ is contained in $R(\beta)^{\Lambda^n}-\gmod$, for large enough $n$. Hence, it is enough to define isomorphisms between functors 
\[
\mathcal{F}_n := \mathcal{F}|_{R(\beta)^{\Lambda^n}-\gmod} \to \mathcal{G}_n := \mathcal{G}|_{R(\beta)^{\Lambda^n}-\gmod}\;,
\]
for each $n$, in a consistent construction.

Let us take the sequence of modules $P_n\in R(\beta)-\gmod$ with homomorphisms
\begin{equation}\label{eq:proj}
\ldots \to_{\phi_n} P_n \to_{\phi_{n-1}} P_{n-1} \to_{\phi_{n-2}} \ldots \to_{\phi_1} P_1 \to_{\phi_0} L_0\;,
\end{equation}
so that each $P_n$ is the projective cover of $L_0$ in $R(\beta)^{\Lambda^n}-\gmod$ (which exists, since $R(\beta)^{\Lambda^n}$ is finite dimensional). 

Applying the functor $\mathcal{F}$ on \eqref{eq:proj}, we may choose vectors $f_n\in \mathcal{F}(P_n)$, $n\geq1$, so that $\mathcal{F}(\phi_n)(f_{n+1}) = f_n$ holds, and $\mathcal{F}(\phi_0)(f_{1}) \neq 0$ ($\mathcal{F}(\phi_0)\neq0$ by exactness).

Let us define a natural transformation 
\[
\tau_n: \Hom_{R(\beta)^{\Lambda^n}}( P_n , \bullet) \to \mathcal{F}_n
\]
of functors going from $R(\beta)^{\Lambda^n}-\gmod$ to $Vec$, by setting $\mathcal{F}_n(T)(f_n)\in \mathcal{F}_n(M)$ for each $T\in \Hom_{R(\beta)^{\Lambda^n}}( P_n , M)$.

Exactness easily implies that $\tau_n(M)$ is an injective map, for all $M\in R(\beta)^{\Lambda^n}-\gmod$. Moreover, taking exactness together with the assumed dimension properties of $\mathcal{F}$ and with the defining properties of a projective cover, leads to an equality of dimensions between $\Hom_{R(\beta)^{\Lambda^n}}( P_n , M)$ and $\mathcal{F}_n(M)$. Hence, $\tau_n$ is an isomorphism of functors.

By arguing the same for $\mathcal{G}_n$, we reach an isomorphism $\mathcal{F}_n \cong  \Hom_{R(\beta)^{\Lambda^n}}( P_n , \bullet) \cong \mathcal{G}_n$ as desired. The consistent choice of $\{f_n\}$ taken for the construction of such functor isomorphisms allows for a limit extension to an isomorphism $\mathcal{F} \cong \mathcal{G}$.

\end{proof}

\begin{lemma}\label{lem:whit}
Let $\mathbf{i}_0\in \mathbb{Z}^t$ be a BZ-sequence for $\beta\in Q_+$. Let $\underline{a} = \underline{a}(\beta)\in A_t$ be the unique element for which $\beta(\mathbf{i}_0, \underline{a}) = \beta$ holds.

Then, for any supercuspidal $\rho\in \irr(G_d)$, we have an isomorphism of functors
\[
\mathbf{W}\circ \mathcal{T}_{\rho, \beta} \circ \fgt \;\cong\;  \fgt \circ \theta_{\mathbf{i}_0}^{\underline{a}(\beta)}\,
\]
from $R(\beta)-\gmod$ to $\Vect$, where $\fgt$ denotes the grading-forgetful functor, and $\Vect$ is naturally identified with $R(0)-\nmod$.

\end{lemma}
\begin{proof}
When taking into account exactness, Proposition \ref{prop:rodier}, Lemma \ref{lem:catg} and the identity \eqref{eq:zel-kr}, we are left with the task of showing that,
for any $\m\in \Mult$ with $\wt(\m)= \beta$, the space $\theta_{\mathbf{i}_0}^{\underline{a}(\beta)}(L_\m)$ is zero, unless $\m$ is generic, and that in the latter case the space is $1$-dimensional.

Let us write $\mathbf{i}_0 = (T, T-1,\ldots, -T)$ and $\underline{a} = (a_{-T}, \ldots, a_T)$.

Note, that $\beta = \sum_{i=-T}^T a_i\alpha_i$ implies $\partial(\beta) = \sum_{i=-T}^T a_i \Delta(i,i)$.

By construction, the proper standard module $\KR(\partial(\beta))$ is written as the product
\[
\Delta(-T, -T)^{\circ a_{-T}} \circ \cdots \circ \Delta(T-1, T-1)^{\circ a_{T-1}}\circ \Delta(T, T)^{\circ a_{T}}\cong
\]
\[
\cong L(-T, a_{-T}) \circ \cdots \circ L(T-1, a_{T-1})\circ L(T,  a_{T})\in R(\beta)-\gmod\;,
\]
where isomorphisms are taken up to a shift of grading.

It then follows by adjunction, that as vector spaces, we have an identification
\[
\Hom_{R(\beta)-\nmod}(\KR(\partial(\beta)), L_\m )\cong \theta_{\mathbf{i}_0}^{\underline{a}(\beta)}(L_\m)\;.
\]
Indeed, by Theorem \ref{thm:kr} the above space is either $1$-dimensional, when $\m = \partial(\beta)$, or zero, otherwise.


\end{proof}

We are now ready to give a similar connection between more general Bernstein-Zelevinsky derivatives on the $p$-adic side, and our derivative operations relative to a BZ-sequence on the quiver Hecke algebra side.

\begin{proposition}\label{prop:bz-commute}
Let $\mathbf{i}_0\in \mathbb{Z}^t$ be a BZ-sequence for $\beta\in Q_+$.

For any $\underline{a}\in A_t$ with $\gamma = \beta(\mathbf{i}_0, \underline{a})\leq \beta$ and supercuspidal $\rho\in \irr(G_d)$, we have an isomorphism of functors
\[
\mathbf{P} \circ \mathfrak{D}_{d|\beta|,d|\gamma|}\circ \mathcal{T}_{\rho, \beta} \circ \fgt \;\cong\;  \mathcal{T}_{\rho, \beta-\gamma}\circ \fgt \circ \theta_{\mathbf{i}_0}^{\underline{a}}\,
\]
from $R(\beta)-\gmod$ to $\mathcal{C}_\rho^{\beta - \gamma}$, where $\fgt$ denotes the grading-forgetful functor and $\mathbf{P}: \mathcal{C}(\rho, |\beta-\gamma|) \to  \mathcal{C}_\rho^{\beta - \gamma}$ is the projection functor.

\end{proposition}

\begin{proof}
Note first, that
\[
\mathbf{P}\circ \mathfrak{D}_{|\beta|,|\gamma|} =  ( \mathbf{W}\otimes \id)\circ \mathbf{r}_{\gamma, \beta-\gamma}\;,
\]
holds as functors from $\mathcal{C}_{\rho}^\beta$, with $\mathbf{r}_{\gamma, \beta-\gamma}$ being defined as in Theorem \ref{thm:bern-rou}.

It follows from definition of divided power functors that the identity
\[
\theta_{\mathbf{i}}^{\underline{a}} =  (\id \otimes \theta_{\mathbf{i}}^{\underline{a}}) \circ \rres_{\beta-\gamma,\gamma},
\]
may be written.

The isomorphism of functors now follows directly from Lemma \ref{lem:whit} and compatibility of $\mathcal{T}_{\rho}$ with restriction functors, as stated in Theorem \ref{thm:bern-rou}.

\end{proof}

The picture becomes clearer when considering the full derivative functor $\mathfrak{D}$ on representations of $p$-adic groups.

Note, that since Jacquet functors take representations in $\mathcal{C}_{\rho}^{\mathbb{Z}}$ into $\mathcal{C}_{\rho}^{\mathbb{Z}}\times \mathcal{C}_{\rho}^{\mathbb{Z}}$, for all supercuspidal $\rho\in \irr(G_d)$, we may view $\mathfrak{D}$ as an endofunctor of $\mathcal{C}_{\rho}^{\mathbb{Z}}$. This allows for a direct comparison of derivative operations defined on both sided of the Bernstein-Rouquier equivalences $\mathcal{T}_{\rho}$.

\begin{corollary}\label{cor:bz-der}

Let $\mathbf{i}_0\in \mathbb{Z}^t$ be a BZ-sequence for $\beta\in Q_+$.

Then, for all $N\in R(\beta)-\gmod$ and supercuspidal $\rho\in \irr(G_d)$, we have an isomorphism
\[
\mathfrak{D}(\mathcal{T}_{\rho}(N)) \cong \mathcal{T}_{\rho}\left(   \oplus_{\underline{a}\in A_t} \theta_{\mathbf{i}_0}^{\underline{a}}(N)  \right)\;
\]
of representations in the category $\mathcal{C}_{\rho}^{\mathbb{Z}}$.

\end{corollary}

\begin{proof}
Since the map $\underline{a} \in A_t\,\mapsto \beta(\mathbf{i}_0, \underline{a})\in Q_+$ is clearly injective, each module $ \theta_{\mathbf{i}_0}^{\underline{a}}(N)$ belongs to a different direct summand of the category $\mathcal{D}$. Thus, the statement amounts to summing up the identity of Proposition \ref{prop:bz-commute}, over all $\underline{a}\in A_t$.

\end{proof}

Finally, the notions of highest non-zero derivatives defined differently on each side of the Bernstein-Rouquier equivalence are revealed to correspond one another.

\begin{proposition}
For a module $M\in \mathcal{D}$ and a choice of a supercuspidal representation $\rho\in \irr(G_d)$, we have an isomorphism
\[
\mathfrak{D}_{BZ}(\mathcal{T}_\rho(M)) \cong \mathcal{T}_{\rho}(\theta_{BZ}(M))\;.
\]

\end{proposition}
\begin{proof}

Let $\mathbf{i}_0\in \mathbb{Z}^t$ be a BZ-sequence for $\wt(M)$. Let $a_M = a(\mathbf{i}_0, M)\in A_t$ be the $\mathbf{i}_0$-derivative of $M$.

By Corollary \ref{cor:bz-der}, it is enough to show that $|\beta(\mathbf{i}_0, \underline{a})| < |\beta(\mathbf{i}_0, a_M)|$ holds, for any $\underline{a}\in A_t$ with $\underline{a}< a_M$ and $\theta_{\mathbf{i}_0}^{\underline{a}}(M)\neq \{0\}$.

By exactness we may assume that $M$ is a simple (and self-dual) module, i.e. $M \cong L_{\m}$, for $\m = \sum_{i=1}^k \Delta_i \in \Mult$ (with $\Delta_i\in \Seg$). Now, by Proposition \ref{prop:BZ} $|\beta(\mathbf{i}_0, a_M)| = |\mathfrak{b}(\m)|= k$.

Since $L_{\m}$ is a quotient module of $\KR(\m)$, it follows from Proposition \ref{prop:filt-der} that there exists a decomposition $\underline{a} = \underline{a}_1 + \ldots + \underline{a}_{k}$ in $A_t$, such that $\theta^{\underline{a}_i}_{\mathbf{i}_0}(L_{\Delta_i})$ is non-zero for each $i=1,\ldots, t$.

Again from Proposition \ref{prop:BZ} we must have $\beta(\mathbf{i}_0, \underline{a}_i)\in \{0,\alpha_{b(\Delta_i)}\}$. In particular, we have
\[
|\beta(\mathbf{i}_0, \underline{a})| = \sum_{i=1}^k |\beta(\mathbf{i}_0, \underline{a}_i)| \leq k\;.
\]
Since $\beta(\mathbf{i}_0, \cdot)$ is injective, equality above will not hold unless $ \underline{a} = a_M$.

\end{proof}

\section{Normal sequences}\label{sect:normal}

The theory of square-irreducible modules (also commonly known as real modules) for quiver Hecke algebras, and normal sequences of those modules, as developed in \cite{kkko-mon,kkko0,kk19}, came to be a useful tool for many aspects of the related representation theory. In particular, normal sequences are a convenient mechanism for producing reducible modules whose submodule lattice is well-structured, and yet non-trivial.

Let us recite some basics of this theory.

A simple module $L\in \irr_{\mathcal{D}}$ is said to be \textit{square-irreducible}, if $L\circ L$ is a simple module.

Given a square-irreducible module $L$ and any $M\in \irr_{\mathcal{D}}$, the product $L\circ M$ has a simple head $N$ (\cite[Theorem 3.2]{kkko0}).
The integer $\widetilde{\Lambda}= \widetilde{\Lambda}(L,M)\in \mathbb{Z}$, for which $N\langle \widetilde{\Lambda}(L,M)\rangle$ is self-dual is a useful invariant to be recorded.

For future convenience, we also set $\widetilde{\Lambda} (L, M\langle k\rangle) = \widetilde{\Lambda} (L, M)$, for any shift $k\in \mathbb{Z}$.

Given a square-irreducible module $L\in \irr_{\mathcal{D}}$ and an admissible sequence $\mathbf{i}\in \mathbb{Z}^t$, it follows from Proposition \ref{prop:hst-der} that $\theta_{\mathbf{i}}(L)$ remains square-irreducible (see also \cite[Corollary 10.1.6]{kkko-mon}).

\begin{lemma}\label{lem:der-lambda}
Let $L, M\in \irr_{\mathcal{D}}$ be given modules, with $L$ being square-irreducible. Suppose that the simple head $N$ of $L\circ M$ satisfies
\[
 a(\mathbf{i},N) =  a(\mathbf{i},L) +  a(\mathbf{i},M)\;.
\]
Then,
\[
\widetilde{\Lambda}(\theta_{\mathbf{i}}(L), \theta_{\mathbf{i}}(M)) = \widetilde{\Lambda}(L,M) + \Phi(\mathbf{i},L,M)\;.
\]
\end{lemma}

\begin{proof}
Arguing as in the proof of Proposition \ref{prop:hst-der}, we see that
\[
\theta_{\mathbf{i}}(L\circ M) = \theta_{\mathbf{i}}^{a(\mathbf{i}, L) + a(\mathbf{i}, M)}(L\circ M)\;.
\]
The assumption further implies that $\theta_{\mathbf{i}}(N) = \theta_{\mathbf{i}}^{a(\mathbf{i}, L) + a(\mathbf{i}, M)}(N)$. Hence, by functor exactness and Proposition \ref{prop:hst-der} again, we see that $\theta_{\mathbf{i}}(N)$ is the simple head of $\theta_{\mathbf{i}}(L)\circ \theta_{\mathbf{i}}(M)\langle \Phi(\mathbf{i},L,M) \rangle$.

By Lemma \ref{lem:simple-sd}, $\theta_{\mathbf{i}}(N)\langle \widetilde{\Lambda}(L,M) \rangle$ remains self-dual.

\end{proof}

For any non-zero modules $M_1,M_2\in \mathcal{D}$, there is a well defined non-zero \textit{$R$-matrix} intertwiner
\[
R_{M_1,M_2}:M_1\circ M_2 \to M_2\circ M_1
\]
in $R(\wt(M_1)+ \wt(M_2))-\nmod$ (see e.g. \cite[Section 2.2]{kkko-mon}). In fact, there is an integer $\Lambda = \Lambda(M_1,M_2)$, for which $R_{M_1,M_2}:M_1\circ M_2 \to M_2\circ M_1\langle -\Lambda \rangle$ is a morphism of graded modules in $R(\wt(M_1)+ \wt(M_2))-\gmod$.

Following \cite{kk19}, we say that a tuple $(L_1,\ldots, L_k)$ of square-irreducible modules in $\irr_{\mathcal{D}}$ is a \textit{normal sequence}, if the composition of $R$-matrices
\[
R_{L_1,\ldots,L_k} := R_{L_{k-1},L_k}\circ\cdots\circ (R_{L_2,L_k} \circ \cdots \circ R_{L_2,L_3})\circ (R_{L_1,L_k} \circ \cdots \circ R_{L_1,L_2})
\]
does not vanish. Here, the composition is viewed as a map from $L_1\circ \cdots \circ L_k$ to $L_k\circ\cdots \circ L_1$, where each composed $R$-matrix is viewed as an intertwiner, when tensored with the identity map on the other product components.

\begin{proposition}\label{prop:simplehd}\cite[Lemma 2.6]{kk19}
For a normal sequence $(L_1,\ldots, L_k)$, the product $L_1\circ \cdots \circ L_k$ has a simple head, given by the image of $R_{L_1,\ldots,L_k}$.

\end{proposition}

\begin{lemma}\label{lem:add-desc}
 Let $(L_1,\ldots, L_k)$ be a normal sequence, and $H$ be the simple head of the product $L_1\circ \cdots \circ L_k$.
Suppose that the equality
 \[
 a(\mathbf{i},H) = a(\mathbf{i},L_1) + \ldots + a(\mathbf{i},L_k)
 \]
 holds for a given admissible sequence $\mathbf{i}\in \mathbb{Z}^t$.

Then,
\[
 a(\mathbf{i},N_{i,j}) = a(\mathbf{i},L_i) + a(\mathbf{i},L_j)
 \]
holds for all $1\leq i<j\leq k$, where $N_{ij}$ denotes the simple head of $L_i\circ L_j$.
\end{lemma}

\begin{proof}
Assume the contrary for some $1\leq i<j\leq k$. Hence,  $a(\mathbf{i},N_{i,j}) < a(\mathbf{i},L_i) + a(\mathbf{i},L_j)$.

Recall that the $R$-matrix $R_{L_i,L_j}$, whose image is $N_{i,j}$, appears (tensored with identity maps) in the composition of $R_{L_1,\ldots,L_k}$. Thus, $H = Im\; R_{L_1,\ldots,L_k}$ must appear as a subquotient of a convolution product involving the modules $(N_{i,j}, (L_r)_{r\neq i,j})$ in a certain order. By Lemma \ref{lem:add-string}, that would mean
\[
a(\mathbf{i},H) \leq  a(\mathbf{i},N_{i,j}) + \sum_{r\neq i,j} a(\mathbf{i},L_r) <  a(\mathbf{i},L_1) + \ldots + a(\mathbf{i},L_k)\;,
\]
contradicting the assumption.

\end{proof}

\begin{proposition}\label{prop:cond-equiv}\cite[Lemma 2.7]{kk19}\cite[Proposition 7.4]{gur-klrrsk}
Let $(L_1,\ldots, L_k)$ be a tuple of square-irreducible modules in $\irr_{\mathcal{D}}$.

The following are equivalent:
\begin{enumerate}
  \item\label{itnorm1} The tuple $(L_1,\ldots, L_k)$ is a normal sequence.
  \item\label{itnorm2} The tuple $(L_2,\ldots, L_k)$ is a normal sequence, and the identity
  \[
    \Lambda(L_1, H) = \Lambda(L_1,L_2) + \ldots  + \Lambda(L_1,L_k)
  \]
  holds, for $H\in \irr_{\mathcal{D}}$, such that $H\langle h\rangle$ is the simple head of $L_2\circ\cdots\circ L_k$.
  \item\label{itnorm3} The tuple $(L_2,\ldots, L_k)$ is a normal sequence, and the identity
  \[
    \widetilde{\Lambda}(L_1, H) = \widetilde{\Lambda}(L_1,L_2) + \ldots  + \widetilde{\Lambda}(L_1,L_k)
  \]
  holds, for $H\in \irr_{\mathcal{D}}$, such that $H\langle h\rangle$ is the simple head of $L_2\circ\cdots\circ L_k$.

\end{enumerate}

\end{proposition}

\begin{theorem}\label{thm:der-nrm}
 Let $(L_1,\ldots, L_k)$ be a normal sequence, and $H$ be the simple head of the product $L_1\circ \cdots \circ L_k$.

 Suppose that the equality
 \[
 a(\mathbf{i},H) = a(\mathbf{i},L_1) + \ldots + a(\mathbf{i},L_k)
 \]
 holds for a given admissible sequence $\mathbf{i}\in \mathbb{Z}^t$.

 Then $(\theta_{\mathbf{i}}(L_1),\ldots, \theta_{\mathbf{i}}(L_k))$ is a normal sequence. The simple head of $\theta_{\mathbf{i}}(L_1)\circ \cdots\circ \theta_{\mathbf{i}}(L_k)$ is isomorphic to $\theta_{\mathbf{i}}(H)\langle -\Phi(\mathbf{i},L_1,\ldots, L_k)\rangle$.

\end{theorem}

\begin{proof}
The sequence $(L_2,\ldots, L_k)$ is normal. Let $H'$ be the simple head of $L_2\circ \ldots \circ L_k$. Then, $H$ must appear as the simple head of $L_1\circ H'$. In particular, $ a(\mathbf{i},H) \leq  a(\mathbf{i},L_1) +  a(\mathbf{i},H')$ by  Lemma \ref{lem:add-string}. When combined with the assumed equality, a subsequent application of same lemma gives
\begin{equation}\label{eq:der-nom1}
 a(\mathbf{i},H') = a(\mathbf{i},L_2) + \ldots + a(\mathbf{i},L_k)\;.
\end{equation}
We now prove the statement by induction on the length $k$. Hence, we may assume that $(\theta_{\mathbf{i}}(L_2),\ldots, \theta_{\mathbf{i}}(L_k))$ is a normal sequence and that  $\theta_{\mathbf{i}}(H')\langle -\Phi(\mathbf{i},L_2,\ldots, L_k)\rangle$ is the simple head of the product $\theta_{\mathbf{i}}(L_2)\circ \cdots\circ \theta_{\mathbf{i}}(L_k)$.

The criterion in Proposition \ref{prop:cond-equiv} now gives the equality
\begin{equation}\label{eq:der-nom2}
\widetilde{\Lambda}(L_1,H') = \widetilde{\Lambda}(L_1,L_2) + \ldots + \widetilde{\Lambda}(L_1,L_k)\;,
\end{equation}
and further reduces our task to showing the validity of
\begin{equation}\label{eq:der-nom3}
\widetilde{\Lambda}(\theta_{\mathbf{i}}(L_1),\theta_{\mathbf{i}}(H')) = \widetilde{\Lambda}(\theta_{\mathbf{i}}(L_1),\theta_{\mathbf{i}}(L_2)) + \ldots + \widetilde{\Lambda}(\theta_{\mathbf{i}}(L_1),\theta_{\mathbf{i}}(L_k))\;.
\end{equation}
It follows from Lemmas \ref{lem:add-desc} and \ref{lem:der-lambda} that
\[
\widetilde{\Lambda}(\theta_{\mathbf{i}}(L_1), \theta_{\mathbf{i}}(L_j)) = \widetilde{\Lambda}(L_1,L_j) + \Phi(\mathbf{i},L_1,L_j)\;,
\]
for all $2\leq j\leq k$. Similarly, $\widetilde{\Lambda}(\theta_{\mathbf{i}}(L_1), \theta_{\mathbf{i}}(H')) = \widetilde{\Lambda}(L_1,H') + \Phi(\mathbf{i},L_1,H')$.

Thus, subtracting \eqref{eq:der-nom2}, the desired equality \eqref{eq:der-nom3} reduces to
\[
\Phi(\mathbf{i},L_1,H') = \Phi(\mathbf{i},L_1,L_2)+ \ldots + \Phi(\mathbf{i},L_1,L_k)\;.
\]
Indeed, this additivity property now follows directly from the defining formula for $\Phi$ in Proposition \ref{prop:hst-der}.
\end{proof}

Let us highlight Theorem \ref{thm:der-nrm} through an introduction of the following notion of often-recurring reducible modules.

\begin{definition}

We say that a quiver Hecke algebra module $M\in \mathcal{D}$ is \textit{spearheaded}, if
\begin{enumerate}
  \item The head $H\in \mathcal{D}$ of $M$ is simple and self-dual.
  \item The graded multiplicity $m(M,H)(q)$ is the constant polynomial $1$. In other words, grading-shifts of $H$ appear only once in the Jordan-H\"{o}lder series of $M$.
  \item For any $H\not\cong L\in \irr_{\mathcal{D}}$, the graded multiplicity $m(M,L)(q)$ is a polynomial satisfying $m(M,L)(0) =0$. (i.e. all appearing powers are positive.)
\end{enumerate}

\end{definition}

\begin{proposition}\label{prop:spear-normal}
  For a normal sequence $(L_1,\ldots, L_k)$ in $\mathcal{D}$, there exists an integer $h\in \mathbb{Z}$, such that the product module
  \[
  L_1\circ\cdots\circ L_k\langle h\rangle
  \]
is spearheaded.
\end{proposition}
\begin{proof}
This is \cite[Proposition 7.7]{gur-klrrsk} together with Proposition \ref{prop:simplehd}.
\end{proof}

\begin{corollary}
Let $(L_1,\ldots, L_k)$ be a normal sequence in $\mathcal{D}$. Let $h\in \mathbb{Z}$ be the integer for which $L_1\circ\cdots\circ L_k\langle h\rangle$ becomes spearheaded, and let $H\in \irr_{\mathcal{D}}$ be its simple head.

Suppose that the equality
 \[
 a(\mathbf{i},H) = a(\mathbf{i},L_1) + \ldots + a(\mathbf{i},L_k)
 \]
 holds for a given admissible sequence $\mathbf{i}\in \mathbb{Z}^t$.

Then, the module
\[
\theta_{\mathbf{i}}(L_1\circ \ldots \circ L_k)\langle h\rangle \cong \theta_{\mathbf{i}}(L_1)\circ \cdots \circ \theta_{\mathbf{i}}(L_k)\langle h+\Phi(\mathbf{i},L_1,\ldots, L_k)\rangle
\]
remains spearheaded. Its head is given by $\theta_{\mathbf{i}}(H)\in \irr_{\mathcal{D}}$.

\end{corollary}

\begin{remark}
While Theorem \ref{thm:der-nrm} and Proposition \ref{prop:spear-normal} give the positivity property of graded multiplicities inside a derived product arising from a normal sequence, one can further apply Corollary \ref{cor:gr-mul} for a finer information on those multiplicities.
\end{remark}

Let us observe the case of BZ-derivatives in the context of normal sequences.

\begin{theorem}\label{thm:bz-final}
Let $(L_{\m_1},\ldots, L_{\m_k})$ be a normal sequence given by multisegments $\m_1,\ldots,\m_k\in\Mult$.
Suppose that $h\in \mathbb{Z}$ is such that the module
\[
M = L_{\m_1}\circ\cdots\circ L_{\m_k}\langle h\rangle\in\mathcal{D}
\]
is spearheaded, with a head isomorphic to $L_{\m}$, for $\m\in \Mult$.

If the equality
\[
\mathfrak{b}(\m) = \mathfrak{b}(\m_1)+ \ldots + \mathfrak{b}(\m_k)
\]
holds, then $(L_{\m'_1},\ldots, L_{\m'_k})$ is another normal sequence.

The module
\[
M' = L_{\m'_1}\circ\cdots\circ L_{\m'_k} \langle h +\Phi(\m_1,\ldots,\m_k)\rangle \in \mathcal{D}
\]
is a spearheaded module, whose head is isomorphic to $L_{\m'}$.

For each $L = L_\n \in \irr_{\mathcal{D}}$ with non-zero $m(M, L)(q)$, we have an equality of graded mutlipicities
\[
m(M', L_{\n'})(q) = \left\{ \begin{array}{ll} m(M, L)(q) & \mathfrak{b}(\m) = \mathfrak{b}(\n) \\ 0 &\mathfrak{b}(\m) \neq \mathfrak{b}(\n)\end{array}\right.\;,
\]
and all simple subquotients of $M'$ appear in this manner, i.e., are of the form $\theta_{BZ}(L)$, for a simple subquotient $L$ of $M$.

\end{theorem}

\begin{proof}

This is a consequence of Theorems \ref{thm:der-nrm} and \ref{thm:mult-bz}, following the same argument as in the proof of  Theorem \ref{thm:mult-bz}.

\end{proof}



\subsection{Cyclic sequences}
We would like to import into the quiver Hecke algebra setting some of the notions studied in \cite{minggur-cyclic} under the name of \textit{cyclic} representations, following the analogous concept in the representation theory of quantum affine algebras (\cite{hernan-cyc}).

Let us first explicate some well known facts on the nature of quiver Hecke algebra modules.

\begin{lemma}\label{lem:facts}
For given modules $M_1,\ldots,M_k$ in $\mathcal{D}$ and any permutation $\omega\in \mathfrak{S}_k$, the following statements hold:
  \begin{enumerate}
    \item\label{eq:eqat1} There is an equality
    \[
    m(L,M_1\circ\cdots\circ M_k)(1) =  m(L,M_{\omega(1)}\circ\cdots\circ M_{\omega(k)})(1)
    \]
    of graded multiplicities evaluated at $q=1$, for every $L\in \irr_{\mathcal{D}}$.
    \item\label{eq:exst} There exists a non-zero homomorphism of graded modules
    \[
    M_1\circ\cdots\circ M_k\langle d\rangle \to M_{\omega(1)}\circ\cdots\circ M_{\omega(k)},\;
    \]
  for some shift of grading $d\in \mathbb{Z}$.
  \end{enumerate}
\end{lemma}
\begin{proof}
The existence statement of \eqref{eq:exst} follows from a straightforward generalization of the theory used to construct $R$-matrices. We refer to the proof of \cite[Proposition 1.15]{kkkI}, where the construction is written in detail.

For \eqref{eq:eqat1}, we may apply the exactness of convolution product to assume that $M_i$ are simple, for all $1\leq i\leq k$. Furthermore, it would suffice to prove the equality for $k=2$.

Since
\[
m(L,(M_1\circ M_2)^\ast)(q) = m(L,M_1\circ M_2)(q^{-1})
\]
holds by construction, for all $L\in \irr_{\mathcal{D}}$, the equality becomes a consequence of \cite[Theorem 2.2]{MR2822211}.\footnote{Alternatively, we can apply deeper categorification theory for quiver Hecke algebras, from which it would follow that the Grothendieck ring of monoidal category $\oplus_{\beta\in Q_+} R(\beta)-\nmod$ is commutative.}

\end{proof}

We say that a normal sequence $(L_{\m_1},\ldots, L_{\m_k})$ is \textit{cyclic}, if the simple head of the product $L_{\m_1}\circ\cdots \circ L_{\m_k}$ may be written as $L_{\m_1+\ldots+\m_k}\langle h\rangle$, for a shift $h\in \mathbb{Z}$.

It appears that the proper standard modules of Kleschev-Ram enjoy the following universality property with respect to cyclic normal sequences, which is to be expected when extending notions from the (ungraded) $p$-adic setting.

\begin{proposition}\label{prop:cyclic}
Suppose that $(L_1,\ldots,L_k)$ is a cyclic normal sequence, whose simple head is given by $L_{\m}\langle h\rangle$, with $\m\in \Mult$ and $h\in \mathbb{Z}$.

Then, the product module $L_1\circ \cdots\circ L_k$ is isomorphic, in $R(\wt(\m))-\gmod$, to a quotient of the grading-shifted proper standard module $\KR(\m)\langle h\rangle$.

\end{proposition}
\begin{proof}

Suppose that $L_i\cong L_{\m_i}$, for all $1\leq i \leq k$, and  multisegments $\m_1,\ldots,\m_k\in \Mult$. By assumption, we have $\m = \m_1+\ldots+\m_k$.

By definition, the modules $\KR(\m)$ and $N = \KR(\m_1)\circ \cdots \circ \KR(\m_k)$ are constructed as a product of the same collection of segment modules, taken in different orders, and shifted by different degrees.

It follows from Lemma \ref{lem:facts} that there is a non-zero homomorphism $T:\KR(\m)\langle d\rangle \to N$ of graded modules and that
\begin{equation}\label{eq:LN1}
m(L_{\m},N)(1)=  m(L_{\m}, \KR(\m))(1) = 1\;.
\end{equation}
Now, by description of each $L_{\m_i}$, we have a surjective homomorphism $S:N\to L_1\circ \cdots\circ L_k$. Because of \eqref{eq:LN1}, we see that
\begin{equation}\label{eq:LN2}
m(L_{\m},\ker S)(q)=0\;.
\end{equation}

Since we know that both $\KR(\m)\langle h\rangle$ and  $L_1\circ \cdots\circ L_k$ have simple heads isomorphic to $L_{\m}\langle h\rangle$ and that both multiplicities $m(L_{\m}, \KR(\m)\langle h\rangle)(q)$ and $m(L_{\m}, L_1\circ \cdots\circ L_k)(q)$ are equal to $q^h$, we must have that either the composed homomorphism
\[
\KR(\m)\langle d\rangle \;\to_T \;N\;\to_S L_1\circ \cdots\circ L_k
\]
is surjective and $d=h$, or $S\circ T=0$.

The latter situation would have implied that the non-zero image of $T$ satisfies, through the equality \eqref{eq:LN2}, $m(L_{\m},Im T)(q)=0$. This statement contradicts head-simplicity of $\KR(\m)$.

\end{proof}

\section{RSK standard modules}

We would like to apply the results of previous sections regarding products of quiver Hecke algebra modules on the case of the family of modules constructed in \cite{gur-lap} (for $p$-adic groups) and \cite{gur-klrrsk}, which is based on a Robinson-Schensted-Knuth transform of multisegments.

\subsection{A bitableau construction}
We say that a tuple $\mu = (\mu_1,\ldots,\mu_t)$ of integers $\mu_1\geq \ldots \geq\mu_t \geq1$ with $\mu_1+ \ldots + \mu_t= d$ is a partition of $d$ (or, $d$ may be implicit). We write $|\mu|=d$, and $\ell(\mu) = t$ for the length of the partition. As customary, for $\ell(\mu) < i$, we write $\mu_i=0$.

For a partition $\mu$, we write $\mu^\dagger$ for the conjugate partition of $\mu$.

For any partition $\mu$, we write the numeric invariant
\[
a(\mu) = \sum_{j=1}^{u} \mu_j^\ast(\mu_j^\ast-1)\;,
\]
where $\mu^\dagger = (\mu^\ast_1,\ldots,\mu^\ast_u)$.

Given a partition $\mu = (\mu_1,\ldots,\mu_t)$, we say that a collection of integers $P = (c_{i,j})$, given graphically as

\ytableausetup{mathmode,boxsize=2em}
\[
P=\begin{ytableau}
c_{1,1} & c_{1,2} & \dots & c_{1,\mu_2}
& \dots & c_{1,\mu_1} \\
c_{2,1}    & c_{2,2} & \dots
& c_{2,\mu_2} \\
\vdots & \vdots
& \vdots \\
c_{t,1} & \dots & c_{t,\mu_t}
\end{ytableau}
\]
is an (inverted) \textit{semistandard Young tableaux} of shape $\mu$, if its rows are given by strictly descending integers, while its columns are given by weakly descending integers.

For a given tableau $P = (c_{i,j})$, we define $P' = (c_{i,j}+1)$ to be the tableau obtained by increasing each of its entries by $1$.

We say that a pair of semistandard Young tableaux $(P, Q) = ((c_{i,j}), (d_{i,j}))$ of same shape is \textit{admissible}, if it satisfies  $c_{i,j}\leq d_{i,j}$ for all indices $(i,j)$. We further say that $(P,Q)$ is \textit{permissible}, when $(P',Q)$ remains admissible.

Let us fix an admissible pair of semistandard Young tableaux $(P, Q) = ((c_{i,j}), (d_{i,j}))$ of shape $\mu = (\mu_1,\ldots,\mu_t)$ for the rest of this subsection.

\subsubsection{Bitableau modules}
We would like to produce a quiver Hecke algebra module in $\mathcal{D}$ out of the data given in $(P,Q)$.

Given segments $\Delta_1, \Delta_2\in\Seg$, we write $\Delta_1\smlr\Delta_2$, if $b(\Delta_1)<b(\Delta_2)$
and $e(\Delta_1)<e(\Delta_2)$ hold. This is a strict partial order on $\Seg$.

In these terms, we say that a multisegment
\[
0\neq \la=\sum_{i=0}^k\Delta_i \in \Mult
\]
is a \textit{ladder multisegment}, if $\Delta_i\smlr\Delta_{i-1}$, for $i=1,\ldots,k$.

We write $\Lad\subset \Mult$ for the collection of all ladder multisegments.

For each integer $1\leq i\leq t$, we produce a ladder multisegment by taking
\[
\la_i(P,Q) = \Delta(c_{i,1}, d_{i,1}-1) + \ldots + \Delta(c_{i,\mu_i}, d_{i,\mu_i}-1)\in \Lad\,
\]
where the convention $\Delta(k,k-1)= 0\in \Mult$ is taken.


Hence, each such admissible pair produces a sequence of ladder multisegments
\[
\underline{\la}(P,Q)= (\la_1(P,Q),\ldots,\la_{t}(P,Q)) \in \Lad^{t}\;.
\]

Let us  write the number
\[
C(P,Q) = \#\{(i,j)\;:\; d_{i,j} = c_{i',j'},\mbox{ for }(i',j')\mbox{ with } i<i'\}\;.
\]

Clearly, when $(P,Q)$ is permissible, we have
\begin{equation}\label{eq:cpq}
C(\underline{\la}(P,Q)) = C(P,Q)\mbox{ and }C'(\underline{\la}(P,Q)) = C(P',Q)\;,
\end{equation}
when comparing to the numbers defined in equations \eqref{eq:c-mults} and \eqref{eq:cc-mults}. Note, that the admissible condition alone does not guarantee the former equality.

Finally, we are ready to define the graded module
\[
\Gamma(P,Q) = L_{\la_1(P,Q)} \circ \cdots \circ L_{\la_{\omega}(P,Q)} \langle a(\mu) - C(P,Q)\rangle\in \mathcal{D}\;.
\]

A reasoning for the chosen shift normalization will become evident later in our discussion.

\subsection{RSK construction}

Let us recall the construction of the RSK-standard modules and their known properties.

For any $0\neq \m \in \Mult$, we set its \textit{width} $\omega(\m)$ to be the minimal number of ladder multisegments $\la_1, \ldots, \la_{\omega(\m)}\in\Lad$, for which we can decompose as $\m = \la_1 + \ldots + \la_{\omega(\m)}$.

We write
\[
\Mult \setminus \{0\} = \bigcup_{i=1}^\infty \Mult_i\;,\quad  \Mult_i = \{0\neq \m \in\Mult\;:\; \omega(\m)= i\}  \;.
\]
Note, that $\Lad = \Mult_1$.

Suppose that a multisegment
\[
0\neq \m = \sum_{i\in I}\Delta_i\in \Mult
\]
and a ladder multisegment
\[
\la =  \sum_{j\in J} \Delta_{j} \in \Lad
\]
are given (we take $I$ and $J$ as disjoint index sets). We may write $J = \{j_1,\ldots j_l\}$, with $\Delta_{j_{l}} \smlr\ldots\smlr \Delta_{j_1}$.

We say that the pair
\[
(\la,\m)\in \Lad \times \Mult
\]
is \textit{permissible}, if for every choice of indices $i_1,\ldots, i_m\in I$, for which $\Delta_{i_m} \smlr \ldots \smlr \Delta_{i_1}$ holds (sub-ladder of $\m$), there is an injective increasing function
\[
\phi: \{1,\ldots,m\} \to \{1,\ldots,l\}\;,
\]
for which $b(\Delta_{j_{\phi(t)}}) \leq e(\Delta_{i_t}) \leq e(\Delta_{j_{\phi(t)}})$ holds.




Let $\pairs \subset \Lad \times \Mult$ be the collection of permissible pairs.

In further refinement, we write $\pairs = \bigcup_{i=1}^\infty \pairs_i$, where $\pairs_i \subset \Lad\times \Mult_i$ are the permissible pairs $(\la,\m)$, with $\omega(\m)=i$.

\begin{proposition}\cite[Proposition 2.4]{gur-lap}
There is a bijection
\[
\Vien:\Mult\setminus\{0\} \rightarrow\pairs\;,
\]
which is explicitly given by the Knuth-Viennot implementation of the RSK correspondence.

The inclusion $\Vien(\Mult_i) \subset \pairs_{i-1}$ holds.
\end{proposition}


\subsubsection{RSK algorithm}
Let us briefly recall the precise algorithmic definition of $\Vien$, as appeared in \cite[Section 2.2.2]{gur-lap} (where further detail is available). The reader may choose to skip this part, as our current treatment will largely not require the actual description of the map.

Suppose that
\[
0\neq \n = \sum_{i\in K}\Delta_i\in \Mult
\]
is a given multisegment, with $\Delta_i\in \Seg$, for all $i\in K$.

We define the \emph{depth} of any $i\in K$, to be given by the formula
\begin{align*}
\depth(i)=\max\{j:\exists i_0=i,i_1,\dots,i_j\in K\text{ such that }&\Delta_{i_{r}}\smlr\Delta_{i_{r+1}}
,r=0,\dots,j-1\}\;.
\end{align*}
Thus, a function $\depth=\depth_{\m}:K\rightarrow\Z_{\ge0}$ is defined. We also write $d(\n)=\max_{i\in K}\depth(i)$.

For any $k=0,\dots,d(\n)$, we can always choose an enumeration $\{i^k_1,\dots,i^k_{l_k}\} = \depth^{-1}(k)\subset K$, for which
\[
b(\Delta_{i^k_1}) \leq \ldots \leq b(\Delta_{i^k_{l_k}}) \leq e(\Delta_{i^k_{l_k}})\leq \ldots \leq e(\Delta_{i^k_1})
\]
holds. We set $j_k:=i^k_{l_k}\in K$.

Having made all enumerations, we write the index sets $J = \{j_0,\dots,j_d\}$ and $I = K \setminus J$.

A permutation $i\mapsto i^\vee$ is then defined on the set $K$, whose cycle decomposition is given by $\{(i^k_1,\dots,i^k_{l_k})\}_{k=0,\dots,d}$.

The admissible pair $\Vien(\n) = (\m,\la)\in \pairs$ is now defined as
\[
\la = \sum_{i\in J}\Delta^\ast_i,\quad \m=\sum_{i\in I}\Delta^\ast_i\in \Mult\;,
\]
where $\Delta^\ast_i = \Delta(b(\Delta_i),e(\Delta_{i^\vee}))\in \Seg$, for all $i\in K$.

\subsubsection{Module construction}

We recall the following result, which is the main tool used to connect the RSK algorithm with our setting, and specifically, the Kleshchev-Ram classification.

\begin{theorem}\label{thm:gur-lap}\cite[Theorem 4.3]{gur-lap}
Given a multisegment $0\neq \n\in \Mult$, let us write $\Vien(\n) = (\la,\m)\in \pairs$.

The product $L_{\la}\circ L_{\m}$ has a simple head, whose isomorphism class, up to a shift of grading, is given by $L_\n$.
\end{theorem}

\begin{remark}
For a ladder multisegment $\la\in \Lad$, the simple module $L_\la$ belongs to a class of \textit{homogeneous} modules (see \cite[Section 4.3]{gur-klrrsk}). In particular, those are known to be square-irreducible. Hence, Theorem \ref{thm:gur-lap} may be viewed as an explication of favorable cases of the general phenomenon (see Section \ref{sect:normal}) of head simplicity of $L\circ M$, for square-irreducible $L$ and any simple $M$.

A thorough study of other perspectives on such explications was made in the $p$-adic setting in the works of Lapid-M\'{i}nguez and Aizenbud-Lapid-M\'{i}nguez \cite{LM2,LM-conj,alm}.
\end{remark}

Given a multisegment $\m\in \Mult_\omega$, we may apply the map Knuth-Viennot map recursively:
\[
\Vien(\m)= (\la_1,\m_1),\; \Vien(\m_1) = (\la_2, \m_2),\;\ldots , \Vien(\m_{\omega-2}) = (\la_{\omega-1}, \la_{\omega}) \in \pairs_1 \subset \Lad\times \Lad\;.
\]
We take the resulting $\omega$ ladder multisegments
\[
\RSK(\m) = (\la_1, \la_2,\ldots,\la_\omega)\in \Lad^{\omega}
\]
as the \textit{RSK-transform} of $\m$.

The properties of $\Vien$ and $\RSK$ imply that the resulting tuple of multisegments may be described as
\[
\RSK(\m) = \underline{\la}(P_\m,Q_\m)\;,
\]
for a pair of semistandard Young tableaux $(P_\m,Q_\m)$ of same shape\footnote{In fact, viewing elements $\Mult$ as multisets of pairs of integers, the bitableaux $(P_\m,Q_\m)$ would be the original (and better familiar) definition of the RSK-transform.}. Moreover, by \cite[Proposition 2.4]{gur-lap} $(P_\m,Q_\m)$ may assumed to be a permissible pair (a condition that clearly makes the pair unique for a given $\m$).


In our current terminology, the \textit{RSK-standard module} attached to $\m \in \Mult$ is defined to be
\[
\Gamma(\m) = \Gamma(P_\m, Q_\m)\in \mathcal{D}\;.
\]

The main results of \cite{gur-lap,gur-klrrsk} position RSK-standard modules as a natural realization method of any simple module in $\irr_{\mathcal{D}}$ given through its Kleshchev-Ram parameter.

\begin{theorem}\label{thm:spear-rsk}\cite[Section 7.1]{gur-klrrsk}
For any multisegment $\m\in \Mult$, the tuple $\RSK(\m) = (\la_1, \la_2,\ldots,\la_\omega)$ produces a normal sequence of modules $(L_{\la_1},\ldots,L_{\la_\omega})$.

The module $\Gamma(\m)$ is spearheaded, and its head is isomorphic to $L_\m$.
\end{theorem}

\subsection{Derived RSK}

Given a multisegment $\m\in\Mult$, due to permissibility of $(P_\m, Q_\m)$, we may further define the \textit{derived RSK module} as
\[
\Gamma'(\m) = \Gamma(P'_\m,Q_\m)\;.
\]

\begin{theorem}\label{thm:head-der}
For any multisegment $\m\in\Mult$, the derived RSK module $\Gamma'(\m)$ is spearheaded with a head isomorphic to $L_{\m'}$.

For each $L = L_\n \in \irr_{\mathcal{D}}$ with non-zero $m(\Gamma(\m), L)(q)$, we have an equality of graded mutlipicities
\[
m(\Gamma'(\m), L_{\n'})(q) = \left\{ \begin{array}{ll} m(\Gamma(\m), L)(q) & \mathfrak{b}(\m) = \mathfrak{b}(\n) \\ 0 &\mathfrak{b}(\m) \neq \mathfrak{b}(\n)\end{array}\right.\;,
\]
and all simple subquotients of $\Gamma'(\m)$ appear in this manner, i.e., are of the form $\theta_{BZ}(L)$, for a simple subquotient $L$ of $\Gamma(\m)$.



\end{theorem}

\begin{proof}
It is a known feature of the RSK correspondence (see \cite[Proposition 5.3]{gur-klrrsk}) that $\mathfrak{b}(\m) = \sum_{j=1}^{\omega(\m)}\mathfrak{b}(\la_j(P_{\m},Q_{\m}))$ holds.

Hence, the required condition in the statement of Theorem \ref{thm:bz-final} is satisfied in the case of the normal sequence produced by Theorem \ref{thm:spear-rsk} for $\m$. The theorem now follows from Theorem \ref{thm:bz-final}, once noting equation \eqref{eq:cpq} and the equality in Proposition \ref{prop:combi}.

\end{proof}

Let us adopt a reverse point of view on the wide family of derived RSK modules, by listing them according the to isomorphism class of their simple head.

Recall (Section \ref{sect:compare}) that $\partial: Q_+\to \Mult$ is the additive monoid homomorphism sending each $\alpha_i\in \mathcal{I}$ to $\alpha_i = \Delta(i,i)\in \Mult$.

Note, that for any $\gamma\in Q_+$ and any $\m\in \Mult$, we have
\[
(\m^+ + \partial(\gamma))' = \m\;.
\]
In fact, the above equation clearly describes all multisegments $\n\in \Mult$, for which $\n' = \m$ holds, for a given $\m\in \Mult$ (``antiderivatives").

For any $\gamma\in Q_+$ and any $\m\in \Mult$, we define
\[
\Gamma(\m, \gamma): = \Gamma'(\m^+ + \partial(\gamma))\in\mathcal{D}\;.
\]

According to Theorem \ref{thm:head-der}, each $\Gamma(\m,\gamma)$ has a simple head isomorphic to $L_\m$. For a given multisegment $\m\in \Mult$, going over all possible $\gamma\in Q_+$, will produce all derived RSK modules whose heads are isomorphic to $L_\m$.

It is an easy exercise on the definition of the RSK algorithm to see that for any $\m\in \Mult$, we have $\RSK(\m) = (\la'_1,\ldots,\la'_\omega)$, where $\RSK(\m^+) = (\la_1,\ldots,\la_\omega)$. It then follows that
\[
\Gamma(\m,0) = \Gamma(\m)\;.
\]



In the $p$-adic setting, the latter point of view was taken in \cite[Section 5]{gur-lap} to construct, through a slightly different approach, a family of modules that was named \textit{enhanced RSK} modules. Indeed, this family coincides (up to the usual categorical equivalences and grading-forgetting) with the derived RSK modules.

\section{Specht modules}
As opposed to the (affine) quiver Hecke algebra setting, explicit constructions of simple modules for cyclotomic quotients tend not to fit well with classifications of the Kleshchev-Ram-Zelevinsky type. Instead, a common theme of module constructions involves the Specht method, tracing back to the classical setting of representations of finite permutation groups.

A \textit{multipartition} $\underline{\mu} = (\mu^1,\ldots,\mu^l)$ of $d\in \mathbb{Z}_{\geq1}$ is a tuple of partitions, for which $|\mu^1| +\ldots + |\mu^l| = d$.

Given a partition $\mu =(\mu_1,\ldots,  \mu_t)$ and a integer $k\in \mathbb{Z}$, we define the \textit{$k$-content} of $\mu$ to be the weight
\[
\cont(k, \mu) = \sum_{i=1}^t \sum_{j=1}^{\mu_i} \alpha_{k  +j-i}\in Q_+\;.
\]

More generally, given a multipartition $\underline{\mu} = (\mu^1,\ldots,\mu^l)$ of $d\in \mathbb{Z}_{\geq1}$ and a multicharge $\kappa = (k_1,\ldots, k_l)$,  we define the \textit{$\kappa$-content} of $\underline{\mu}$ to be the weight
\[
\cont(\kappa, \underline{\mu}) = \cont(k_1,\mu^1) + \ldots + \cont(k_l, \mu^l)\in Q_+\;.
\]
Note, that $|\cont(\kappa, \underline{\mu})|= d$ holds.

For a multicharge $\kappa$ and a weight $\beta\in Q_+$, we write $\mathcal{P}^{\kappa}_{\beta}$ for the set of multipartitions $\underline{\mu}$ satisfying $\cont(\kappa, \underline{\mu}) = \beta$.

For any choice of a multipartition $\underline{\mu}\in \mathcal{P}^{\kappa}_{\beta}$, the construction in \cite{univ-specht} produces a graded module
\[
S_{\kappa}(\underline{\mu}) \in R(\beta)^{\Lambda(\kappa)}-\gmod
\]
called a \textit{row Specht module}.

We will refer to the modules
\[
S_{\kappa}(\underline{\mu})^{\sgn} \in R(\beta^\dagger)^{\Lambda(\kappa^\dagger)}-\gmod
\]
obtained by composing row Specht modules with the sign functor, as \textit{column Specht modules}.

For a partition $\mu =(\mu_1,\ldots,\mu_t)$, recall that $\mu^\dagger =(\mu_1^\ast,\ldots, \mu^\ast_u)$ denotes its usual conjugate partition of same integer.

For a multipartition $\underline{\mu} = (\mu^1,\ldots,\mu^l)$, we write $\underline{\mu}^\dagger = ((\mu^l)^\dagger,\ldots,(\mu^1)^\dagger)$.

It is easily checked that for any $\underline{\mu}\in \mathcal{P}^{\kappa}_{\beta}$, we have $\underline{\mu}^\dagger\in \mathcal{P}^{\kappa^\dagger}_{\beta^\dagger}$.

\begin{proposition}\cite[Theorems 7.25, 8.5]{univ-specht}
For any $\underline{\mu}\in \mathcal{P}^{\kappa}_{\beta}$, the modules $ S_{\kappa}(\underline{\mu})^{\sgn}$ and $ S_{\kappa^\dagger}(\underline{\mu}^\dagger)^\ast$ are isomorphic in $R(\beta^\dagger)^{\Lambda(\kappa^\dagger)}-\gmod$, up to a shift of grading.
\end{proposition}
In other words, up to a (computable) shift of grading, column Specht modules are dual modules of row Specht modules.

\subsubsection{Classifying simple modules}

An multipartition $\underline{\mu} =(\mu^1,\ldots,  \mu^l)$ is called \textit{$\kappa$-restricted} (or, \textit{Kleshchev}) for a multicharge $\kappa = (k_1,\ldots, k_l)$, when
\[
\mu^i_{j + k_i- k_{i+1}} \leq \mu^{i+1}_j
\]
holds, for all $1\leq i\leq t-1$ and $1\leq j$.

We write $\mathcal{RP}^{\kappa}_{\beta}\subset \mathcal{P}^{\kappa}_{\beta}$ for the subset of $\kappa$-restricted multipartitions.

For $\underline{\mu} \in \mathcal{RP}^{\kappa}_{\beta}$, we say that the row Specht module $S_{\kappa}(\underline{\mu})$ is a \textit{restricted Specht module}, and the column Specht module $S_{\kappa}(\underline{\mu})^{\sgn}$ is a \textit{regular Specht module}.

Let us denote the subset $(\mathcal{RP}^{\kappa^\dagger}_{\beta^\dagger})^\dagger \subset \mathcal{P}^{\kappa}_{\beta}$ as \textit{$\kappa$-regular} multipartitions. Thus, regular Specht modules may be viewed as certain shifts of grading of the dual mudules to row Specht modules attached to $\kappa$-regular multipartitions.

\begin{proposition}\label{prop:class-specht}
For any multicharge $\kappa$ and a weight $\beta\in Q_+$, the restricted and regular Specht modules give two bases for the Grothendieck group of the category $R(\beta)^{\Lambda(\kappa)}-\nmod$.

Each $S_{\kappa}(\underline{\mu})$, with $\underline{\mu} \in \mathcal{RP}^{\kappa}_{\beta}$, has a simple self-dual head $D_{\kappa}(\underline{\mu})$.

The map $\underline{\mu}\mapsto D_{\kappa}(\underline{\mu})$ gives a bijection from $\mathcal{RP}^{\kappa}_{\beta}$ to $\irr R(\beta)^{\Lambda(\kappa)}$.

\end{proposition}

\subsubsection{Inflated Specht modules}

For any $\Lambda = \Lambda(k)\in P_+$, $k\in \mathbb{Z}$, of level $1$, the algebra $R_d^{\Lambda}$ becomes isomorphic to the complex group algebra of the permutation group $\mathfrak{S}_{d}$.

We see a parametrization of self-dual simple of modules of $R_d^{\Lambda}$, that is,  the set
\[
\sqcup_{\beta\in Q_+\;, |\beta|=d} \irr(\beta)^{\Lambda(k)}\;,
\]
by row Specht modules $S_k(\mu)$, where $\mu$ is a partition of $d$. It naturally corresponds to the classical parametrization of simple complex $\mathfrak{S}_d$-modules by partitions of $d$.

Moreover, since $R_d^{\Lambda}$ is a semisimple algebra in that case, we see that $S_k(\mu)$ are all simple and self-dual. In particular, $S_k(\mu) \cong S_k(\mu)^\ast \cong S_{-k}(\mu^\dagger)^{\sgn}$ shows that row and column Specht modules are the same family, while the sign bijection between them corresponds to conjugation of partitions.

Let us lift the Specht module construction to the level of the category $\mathcal{D}$ of quiver Hecke algebra modules.

For any multipartition $\underline{\mu}\in \mathcal{P}^{\kappa}_{\beta}$, we write
\[
\widehat{S}_{\kappa}(\underline{\mu}) = \infl (S_{\kappa}(\underline{\mu}))\in R(\beta)-\gmod\;.
\]
We also obtain that
\[
\widehat{S}_{\kappa}(\underline{\mu})^{\sgn} = \infl (S_{\kappa}(\underline{\mu})^{\sgn})\in R(\beta^\dagger)-\gmod\;.
\]

\begin{proposition}\label{prop:specht-mult}\cite[Theorem 8.2]{univ-specht}
For a multicharge $\kappa= (k_1,\ldots,k_l)$ and a multipartition $\underline{\mu} =(\mu^1,\ldots,  \mu^l)\in \mathcal{P}^\kappa_{\beta}$, the graded $R(\beta)$-module $\widehat{S}_{\kappa}(\underline{\mu})$ is isomorphic, up to a shift of grading, to the convolution product
\[
\widehat{S}_{k_1}(\mu^1) \circ\cdots \circ \widehat{S}_{k_l}(\mu^l)\;.
\]

The graded $R(\beta^\dagger)$-module $\widehat{S}_{\kappa}(\underline{\mu})^{\sgn}$ is isomorphic, up to a shift of grading, to the convolution product
\[
\widehat{S}_{-k_1}((\mu^1)^\dagger) \circ\cdots \circ \widehat{S}_{-k_l}((\mu^l)^\dagger)\;.
\]

\end{proposition}


\subsection{Specht modules as derived RSK modules}

Given a partition $\mu =(\mu_1,\ldots,  \mu_t)$ of $d\in\mathbb{Z}_{\geq1}$ and a integer $k\in \mathbb{Z}$, let us define the following multisegment
\[
\m(k,\mu) = \sum_{i=1}^t \Delta( k-\mu_i+i, k+i-1)\in \Mult\;.
\]
It is evident that in fact $\m(k,\mu)$ is a ladder multisegment.


Note, that $\wt(\m(k,\mu)) = \cont(k,\mu^\dagger)$ holds in $Q_+$, and that $|\wt(\m(k,\mu))| =d$.

The following fact, or its slight variants across various type $A$ settings, should be well-known, yet explicit references to it prove difficult to trace in the literature. For completeness, we give a proof here (while still relying on results from external literature).

\begin{proposition}\label{prop:eval}
For any integer partition $\mu$ and a integer $k\in \mathbb{Z}$, we have an isomorphism
\[
\widehat{S}_{k}(\mu^\dagger)\cong L_{\m(k,\mu)}
\]
of modules in $R(\wt(\m(k,\mu)))-\gmod$.



\end{proposition}

\begin{proof}
In \cite[Theorem 6.21]{univ-specht}, an explicit complex basis $\{v^T\}_{T\in St(\mu^\dagger)}$ for the module $\widehat{S}_{k}(\mu^\dagger)$ was constructed. Let us recall some of its properties.

Recall that the partition $\mu^\dagger = (\mu^\ast_1,\ldots,\mu^\ast_u)$ may be graphically described as a subset of $\mathbb{Z}_{>0}\times \mathbb{Z}_{>0}$ consisting of pairs $(a,b)$, for which $b\leq \mu^\ast_a$ holds.

In that interpretation, the set of standard $\mu^\dagger$-tableaux $St(\mu^\dagger)$ consists of all possible bijections $T:\{1,\ldots, |\mu|\} \to \mu^\dagger$ that increase along all rows and columns of $\mathbb{Z}_{>0}\times \mathbb{Z}_{>0}$.

Each $T\in St(\mu^\dagger)$ then defines a residue element
\[
\nu^k(T) = (\nu^T_1,\ldots, \nu^T_{|\mu|})\in \mathcal{I}^{\cont(k,\mu^\dagger)}\;,
\]
by taking $\nu^T_i = k -a_i +b_i$, with $T(i) = (a_i,b_i)$, for all $1\leq i\leq |\mu|$.

The vector $v^T\in \widehat{S}_{k}(\mu^\dagger)$ (still by \cite[Theorem 6.21]{univ-specht}) spans a $1$-dimensional space $\mathfrak{e}(\nu^k(T))(\widehat{S}_{k}(\mu^\dagger))$.

We recall further that in our case all vectors $\{v^T\}_{T\in St(\mu^\dagger)}$ turn out to be of same degree, showing that the graded module $\widehat{S}_{k}(\mu^\dagger)$ is concentrated in a single degree. This fact fits the module into a well studied family of homogenous simple modules \cite{kr1}.

On the other hand, being defined by a ladder multisegment, the simple module $L_{\m(k,\mu)}$ is known to be homogenous as well(\cite[Section 4.3]{gur-klrrsk}). Moreover, it follows from a description of its restrictions functors in \cite[Proposition 4.5]{gur-klrrsk} (tracing back to \cite{LapidKret} and \cite[Theorem 3.6]{kr1}), that for any $\nu\in \mathcal{I}^{\wt(\m(k,\mu))}$, the space $\mathfrak{e}(\nu)(L_{\m(k,\mu)})$ is zero, unless $\nu = \nu^k(T)$, for $T\in St(\mu^\dagger)$.

Since the description in \cite[Proposition 4.5]{gur-klrrsk} or \cite[Theorem 3.6]{kr1} clearly determines\footnote{May also be deduced from a general argument of character injectivity in the sense of \cite[Theorem 3.17]{KLR1}} the isomorphism class of a homogenous simple module $L$ from its collection of non-zero weight spaces $\mathfrak{e}(\nu)(L_{\m(k,\mu)})$, our statement follows.

\end{proof}

\subsubsection{Proper restricted multipartitions}

Note, that for a $(k_1,\ldots,k_l)$-restricted multipartition $\underline{\mu}= (\mu^1,\ldots,\mu^l)$, we have
\[
\ell(\mu^i) - k_i \leq \ell(\mu^{i+1}) - k_{i+1}\;,
\]
for all $1\leq i\leq l-1$. We say that a multipartition $\underline{\mu}\in \mathcal{RP}^{\kappa}_{\beta}$ is \textit{proper}, when the latter $l-1$ inequalities are valid as equalities.

For example, for $\kappa = (2,1,-1)$, the multipartition $\underline{\mu} = ( (4,2,2,2,1), (3,3,2,2), (3,2))$ is proper $\kappa$-restricted, while $\underline{\nu} = ( (4,3,2), (3,3,2), (3,1)))$ is $\kappa$-restricted, but not proper.

\begin{proposition}\label{prop:proper-specht}

Suppose that a proper multipartition $\underline{\mu}= (\mu^1,\ldots,\mu^l)\in \mathcal{RP}^{\kappa}_{\beta}$ is given, for a multicharge $\kappa =(k_1,\ldots,k_l)\in \mathbb{Z}^l$. Denote the multisegment
\[
\m = \m(-k_1, \mu^1) + \ldots + \m(-k_l, \mu^l)\in \Mult\;.
\]
Then, the equality
\[
\RSK(\m) = (\m(-k_1, \mu^1),\ldots, \m(-k_l, \mu^l))
\]
holds in $\Lad^l$, and the identity
\[
\Gamma(\m) \cong \widehat{S}_{\kappa}(\underline{\mu})^{\sgn}
\]
in $R(\beta^\dagger)-\gmod$ follows.

\end{proposition}

\begin{proof}
The first equality is a combinatorial statement on the RSK algorithm. Let us prove it by induction on $l$. Hence, it would suffice to show the equality
\[
\Vien(\m) = (\m(-k_1, \mu^1), \m(-k_2, \mu^1) + \ldots + \m(-k_l, \mu^l))
\]
of permissible pairs.

Setting $\mu^i = (\mu^i_1,\ldots, \mu^i_{\ell(\mu^i)})$, for all $i=1,\ldots, l$, we may write

\[
\m = \sum_{i=1}^l \sum_{j=0}^{\ell(\mu^i)-1} \Delta_{i,j}\;,
\]
where $\Delta_{i,j} = \Delta( -k_i-\mu^i_{\ell(\mu^i)- j}+\ell(\mu^i) - j, -k_i + \ell(\mu^i) - j -1)\in \Seg$.

Let us denote the value $r= \ell(\mu^i) - k_i -1 $, which is independent of $i$ by assumption.

Since $k_i-k_{i+1} = \ell(\mu^i)-\ell(\mu^{i+1})$ holds by same assumption, the $\kappa$-restricted condition on $\underline{\mu}$ implies the inequality
\begin{equation}\label{eq:proper-res}
\mu^i_{\ell(\mu^{i})-j} \leq \mu^{i+1}_{\ell(\mu^{i+1})-j}\;,
\end{equation}
for all $j=0,\ldots, \ell(\mu^{i+1})-1$.

Recall (\cite[Section 2.2.2]{gur-lap}) the depth function $\depth(i,j)\in \mathbb{Z}_{\geq0}$ that is used in the RSK algorithm, defined now for any $i=1,\ldots,l$ and $j=0,\ldots,\ell(\mu^i)-1$.

Since $e(\Delta_{i,j}) = r - j$ holds for all valid indices $i,j$, a moment's reflection on the definition of depth will show that $\depth(i,j) = j$.

Therefore, for any $k=0,\ldots, d(\m)$, the value $e(\Delta_{i,j})$ remains constant on all indices $(i,j)\in \depth^{-1}(k)$. In particular, we fall into the case where the index permutation $(i,j) \mapsto (i,j)^\vee$ of the RSK algorithm is the identity permutation. In other words, $\Vien(\m) = (\la,\widehat{\m})\in \pairs$ must arise from a decomposition $\m = \la + \widehat{\m}$ in $\Mult$.

Now, the inequality \eqref{eq:proper-res} implies the conditions $b(\Delta_{1,j})\geq b(\Delta_{2,j}) \geq\ldots$, for all $j=0,\ldots ,\ell(\mu^1)= d(\m)$. By the RSK algorithm, that puts $\la = \sum_{j=1}^{\ell(\mu^1)} \Delta_{i,j} = \m(-k_1, \mu^1)$, which finishes the proof of the first part.

Propositions \ref{prop:specht-mult} and \ref{prop:eval}, together with the construction of $\Gamma(\m)$, imply now that $\Gamma(\m)$ and  $\widehat{S}_{\kappa}(\underline{\mu})^{\sgn}$ are isomorphic, up to a shift of grading.

By Proposition \ref{prop:class-specht}, the module $S_{\kappa}(\underline{\mu})$, and hence $\widehat{S}_{\kappa}(\underline{\mu})^{\sgn}$, has a self-dual simple head. Since $\Gamma(\m)$ is spearheaded, we see that no shift of grading is in fact required.

\end{proof}

\subsubsection{General case}

For each partition $\mu = (\mu_1,\ldots,\mu_t)$, we write $\mu' = (\mu'_j)$ for the partition obtained by setting $\mu'_j = \mu_j-1$, for all $1\leq j$.

Clearly,
\begin{equation}\label{eq:cut-ladder}
\m(k,\mu)' = \m(k,\mu')
\end{equation}
holds, for all $k\in \mathbb{Z}$.

Similarly, for a multipartition $\underline{\mu}= (\mu^1,\ldots,\mu^l)$, we set $\underline{\mu}'= ((\mu^1)',\ldots,(\mu^l)')$.

It is easily verified that whenever $\underline{\mu}$ is $\kappa$-restricted, for a given multicharge $\kappa$, $\underline{\mu}'$ remains $\kappa$-restricted.

\begin{lemma}\label{lem:padding-mu}
For any $\kappa$-restricted multipartition $\underline{\mu}$, there is a proper $\kappa$-restricted multipartition $\underline{\mu}^+$, such that $(\underline{\mu}^+)' = \underline{\mu}$.
\end{lemma}

\begin{proof}
Suppose that $\underline{\mu} = (\mu^1,\ldots,\mu^l)$ with $\mu^i = (\mu^i_1,\ldots, \mu^i_{\ell(\mu^i)})$, for all $1\leq i \leq l$, and $\kappa= (k_1,\ldots,k_l)$. Set $r = \ell(\mu^l) - k_l$.

For each $1\leq i\leq l$, we write a partition
\[
(\mu^i)^+ = (\mu^i_1+ 1, \ldots, \mu^i_{\ell(\mu^i)}+1 ,1,\ldots, 1)\;,
\]
where the number of $1$'s padded in the tail of the formula is given by $r+k_i$.

The resulting multipartition $\underline{\mu}^+ =((\mu^1)^+,\ldots,(\mu^l)^+  )  $ satisfied the required properties.

\end{proof}

\begin{theorem}\label{thm:specht-rsk}
Suppose that a multipartition $\underline{\mu}= (\mu^1,\ldots,\mu^l)\in \mathcal{RP}^{\kappa}_{\beta}$ is given, for a multicharge $\kappa =(k_1,\ldots,k_l)\in \mathbb{Z}^l$. Denote the multisegment
\[
\m = \m(-k_1, \mu^1) + \ldots + \m(-k_l, \mu^l)\in \Mult\;.
\]

Then, a weight $\gamma\in Q_+$ exists, for which an isomorphism
\[
\Gamma(\m,\gamma) \cong \widehat{S}_{\kappa}(\underline{\mu})^{\sgn}
\]
holds in $R(\beta^\dagger)-\gmod$, between the corresponding derived RSK module and the inflated column Specht module.

In particular, we have the identity $\widehat{D}_{\kappa}(\underline{\mu})^{\sgn}\cong L_{\m}$ in $\irr(\beta^\dagger)$.

\end{theorem}

\begin{proof}
Let $\underline{\mu}^+ = (\mu^{1+}, \ldots, \mu^{l+})$ be the proper $\kappa$-restricted multipartition satisfying $(\underline{\mu}^+)' = \underline{\mu}$, which is supplied by Lemma \ref{lem:padding-mu}.

When denoting
\[
\n = \m(-k_1, \mu^{1+}) + \ldots + \m(-k_l, \mu^{l+})\in \Mult\;,
\]
we see by equation \eqref{eq:cut-ladder} that $\n' = \m$. In other words, there exists $\gamma\in Q_+$, so that $\n = \m^+ + \partial(\gamma)$.

By Proposition \ref{prop:proper-specht}, we know that
\[
\RSK(\n)= (\m(-k_1, \mu^{1+}),\ldots, \m(-k_l, \mu^{l+}))\;.
\]
It follows, again from \eqref{eq:cut-ladder}, that the derived RSK module $\Gamma(\m,\gamma) = \Gamma'(\n)$ is isomorphic, up to a shift of grading, to the product $L_{\m(-k_1, \mu^1)} \circ\cdots\circ L_{ \m(-k_l, \mu^l)}$. The latter module is isomorphic, according to Propositions \ref{prop:specht-mult} and \ref{prop:eval}, to $ \widehat{S}_{\kappa}(\underline{\mu})^{\sgn}$.

Lastly, arguing same is in the proof of Proposition \ref{prop:proper-specht}, the shift of grading is unnecessary due to the fact that $\Gamma(\m,\gamma)$ is spearheaded (Theorem \ref{thm:head-der}).

The identity $\widehat{D}_{\kappa}(\underline{\mu})^{\sgn}\cong L_{\m}$ now follows from Theorem \ref{thm:head-der} and the definition of $D_{\kappa}(\underline{\mu})$.

\end{proof}

\begin{corollary}\label{cor:col-rem}
 For any multipartition $\underline{\mu}\in \mathcal{RP}^{\kappa}_{\beta}$, we have identity of graded modules $\theta_{BZ}(\widehat{S}_{\kappa}(\underline{\mu})^{\sgn}) \cong \widehat{S}_{\kappa}(\underline{\mu}')^{\sgn}$, and of their simple heads $\theta_{BZ}(\widehat{D}_{\kappa}(\underline{\mu})^{\sgn}) = \widehat{D}_{\kappa}(\underline{\mu}')^{\sgn}$.

\end{corollary}


\subsection{Universality of proper standard modules}

Let us take note that when $\RSK(\m) = (\la_1,\ldots,\la_\omega)$ satisfies the equality $\m = \la_1 +\ldots + \la_\omega$ in $\Mult$, Theorem \ref{thm:spear-rsk} states that $(L_{\la_1},\ldots, L_{\la_\omega})$ must be a cyclic normal sequence.

More generally, we may look at the situation of $\n'=\m$, for $\n\in \Mult$ with $\RSK(\n) = (\la_1,\ldots,\la_\omega)$ that satisfies $\m = \la'_1+\ldots + \la'_\omega$. In that case, (the proof of) Theorem \ref{thm:head-der} shows that $(L_{\la'_1},\ldots, L_{\la'_\omega})$ is a normal cyclic sequence.

Indeed, with that view in mind, Theorem \ref{thm:specht-rsk} and Proposition \ref{prop:specht-mult} give the following consequence regrading the Specht construction.

\begin{corollary}\label{cor:specht-normalcyc}
For every $\kappa$-restricted multipartition $\underline{\mu}= (\mu^1,\ldots,\mu^l)$, with $\kappa = (k_1,\ldots, k_l)$, the tuple of modules
\[
(\widehat{S}_{-k_1}((\mu^1)^\dagger), \ldots, \widehat{S}_{-k_l}((\mu^l)^\dagger))
\]
in $\irr_{\mathcal{D}}$ is a cyclic normal sequence.
\end{corollary}

Stirring in Proposition \ref{prop:cyclic}, from our general discussion of cyclic modules, we may now state a curious connection between separate classifications of simple modules.

\begin{proposition}\label{prop:regular-quot}

Any inflated regular Specht module is isomorphic to a quotient module of the proper standard module corresponding to its simple head.

\end{proposition}
\begin{proof}
The shift of grading in the statement of Proposition \ref{prop:cyclic} disappears, when noting that both modules involved are spearheaded.
\end{proof}


\subsection{Remarks on involutive symmetries}

Let us briefly address the issue of appearance of the sign involution, or the seemingly assymetric preference of column (regular) Specht modules over row (restricted) Specht modules, in the statement of Theorem \ref{thm:specht-rsk}.

Our discussion involves three separate model constructions for simple modules in $\mathcal{D}$: The Kleshchev-Ram method of proper standard modules, an inflation of the Specht construction from the cyclotomic setting and the RSK construction. Each of those methods involves some non-canonical choices. Let us clarify the consistencies between those choices.

First, let us recall Remark \ref{rem:cycl} on the variation in definition of cyclotomic quotients. In light of this symmetry, for a given restricted multipartition $\underline{\mu}\in \mathcal{RP}^{\kappa}_{\beta}$, we may treat the module $\widehat{S}_{\kappa}(\underline{\mu})^{\sigma}$ as an alternative definition for an inflated restricted Specht module.

Let us consider the isomorphism $ \eta = \sgn_\beta \circ \sigma_\beta : R(\beta)\cong R(\beta^\dagger)$. We see from Proposition \ref{prop:twist-irr} that the class of proper standard modules is closed under twisting by $\eta$. Thus, by applying $\eta$ we obtain the following version of Proposition \ref{prop:regular-quot}.
\begin{corollary}

Any inflated restricted Specht module, when twisted by the involution $\sigma$, is isomorphic to a quotient module of the proper standard module corresponding to its simple head.

\end{corollary}

Second, let us note that the RSK transform of multisegments implicitly depends on a choice of a total ordering of the simple roots $\mathcal{I}$ (i.e. of $\mathbb{Z}$). Reversing that order will produce the alternative transform $\RSK(\m^\dagger)^\dagger$ in our notations (here, $\dagger$ is extended componentwise to $\Lad^\omega$). The resulting RSK-standard modules will have the form $\Gamma(\m^\dagger)^\eta$, as a model for its head $L_\m\in \irr_{\mathcal{D}}$ (The ungraded variant of this class was noted in \cite[Remark 4.5]{gur-lap} in relation to the Sch\"utzenberger involution).

Thus, Theorem \ref{thm:specht-rsk} may be rephrased as an identification between $\sigma$-twisted restricted Specht modules and $\eta$-twisted derived\footnote{Deriving normal sequences in the sense of Theorem \ref{thm:der-nrm} does not commute with twisting by involutions $\sgn$ or $\sigma$. Yet, natural conjugated versions of $\theta_{\mathbf{i}}^{\underline{a}}$ functors may be introduced.} RSK-standard modules.

Finally, we conjecture that the class of RSK-standard modules remains closed under the sign involution. If that was true, we would obtain by Theorem \ref{thm:specht-rsk} that inflated restricted Specht modules (without need of twisting) are special cases of derived variants of RSK-standard modules.

Note, though, that the last point of view would still put restricted Specht modules as models for construction of $L_{\m}^{\sgn}$ out of the data of the multisegment $\m$. Hence, consistency with the Kleshchev-Ram classification still passes through the order reversal (or, Zelevinsky involution) as in Remark \ref{rem:kl-inv}.

\bibliographystyle{alpha}
\bibliography{propo2}{}

\end{document}